\documentclass[11pt]{article}
\usepackage[letterpaper,margin=1.00in]{geometry}
\usepackage{amscd}
\usepackage{mathrsfs}
\usepackage[IL2]{fontenc}

\usepackage{comment} 

\usepackage{color}
\usepackage{amssymb,amsmath,amsthm,amsfonts,latexsym}
\usepackage[utf8x]{inputenc}
\usepackage{bbm} 

\usepackage{graphicx}
\usepackage[disable]{todonotes}

\usepackage{amsmath,amsfonts,amssymb,amsthm} 
\usepackage{tikz}
\usepackage{hyperref}
\usepackage{thmtools}
\usepackage{thm-restate}

\newcommand{\I}[1]{{\mathbbm #1}}

\newtheorem{theorem}{Theorem}[section]

\newtheorem*{claim*}{Claim}
\newtheorem*{theorem*}{Theorem}
\newtheorem*{definition*}{Definition}

\newtheorem{corollary}[theorem]{Corollary}
\newtheorem{lemma}[theorem]{Lemma}

\newtheorem{claim}[theorem]{Claim}

\newtheorem{proposition}[theorem]{Proposition}

\newtheorem{example}[theorem]{Example}
\newtheorem{definition}[theorem]{Definition}
\newtheorem{question}[theorem]{Question}

\newcommand{\comm}[1]{}
\usepackage[capitalize, nameinlink]{cleveref}
\crefname{theorem}{Theorem}{Theorems}
\crefname{proposition}{Proposition}{Propositions}
\crefname{observation}{Observation}{Observations}
\crefname{lemma}{Lemma}{Lemmas}
\crefname{claim}{Claim}{Claims}
\crefname{problem}{Problem}{Problems}
\crefname{conjecture}{Conjecture}{Conjectures}
\crefname{question}{Question}{Questions}
\crefname{example}{Example}{Examples}
\crefname{fact}{Fact}{Facts}

\newcommand{\fH}{\mathcal{H}}

\newcommand{\fJ}{\mathcal{J}}

\newcommand{\fO}{\mathcal{O}}
\newcommand{\fP}{\mathcal{P}}

\def\leukfrac#1/#2{\leavevmode
               \kern.1em
                \raise.9ex\hbox{\the\scriptfont0 ${}_#1$}
                \hskip -1pt\kern-.1em
                /\kern-.15em\lower.10ex\hbox{\the\scriptfont0 ${}_#2$}}

\makeatletter
\def\diam{\mathop{\operator@font diam}\nolimits}
\makeatother

\usepackage{pdflscape}

\begin{document}
\title{Borel equivalence relations induced by actions of tsi Polish groups}

\author{Jan Grebík \\ \small University of Warwick\\
\small \texttt{jan.grebik@warwick.ac.uk}}

\maketitle

\begin{abstract}
    We study Borel equivalence relations induced by Borel actions of tsi Polish groups on standard Borel spaces.
    We characterize when such an equivalence relation admits \emph{classification by countable structures} using a variant of the $\I G_0$-dichotomy.
    In particular, we find a class that serves as a base for non-classification by countable structures for these equivalence relations under Borel reducibility.
    We use this characterization together with the result of Miller \cite{MillerLacunary} to show that if such an equivalence relation admits classification by countable structures but it is not \emph{essentially countable}, then the equivalence relation ${\I E}^{\I N}_0=\I E_3$ Borel reduces to it.
\end{abstract}

\section{Introduction}

In this paper we study complexity of \emph{Borel} equivalence relations induced by Borel actions of \emph{tsi} Polish groups on standard Borel spaces.
This can be seen as part of the program that studies the formalization of the \emph{isomorphism problem}.
The abstract framework for the study of the isomorphism problem is provided by the so-called \emph{invariant descriptive set theory}.
Namely, the class of structures that we want to study is naturally encoded into a Polish topological space $X$ and the isomorphism relation translates to a definable equivalence relation on $X$.
The main notion that allows to compare various isomorphism problems throughout mathematics is called \emph{Borel reducibility}.
This is just an abstract version of the well-known strategy of assigning invariants to structures in order to distinguish them in various settings.
A classical examples include: (a) topological spaces and fundamental groups or (b) Bernoulli shifts and entropy.

During several decades of study, benchmark examples of equivalence relations and their relationships were discovered, see the books \cite{Gao,Kanovei} for a summary.
A particularly important examples are so-called \emph{orbit equivalence relations}.
That is, equivalence relations induced by group actions.
One line of research is to understand all possible complexities of equivalence relations that are induced by actions of groups from a given class.
In this paper we focus on tsi Polish groups and their actions.
Recall that a Polish group is \emph{tsi} if it admits \emph{conjugacy invariant} open base at the identity.
It is well-known that the condition is equivalent with existence of a two-sided invariant compatible metric.
This class includes separable Banach spaces, or, in general, commutative Polish groups.
Recently, tsi Polish groups attracted quite a lot of attention from various perspectives \cite{DingGao,MillerLacunary,AllisonPanagiotopoulos,Allison}.
In this paper, we study possible complexities of equivalence relations induced by actions of tsi Polish groups.
In particular, we focus on the situation when they admit \emph{classification by countable structures} or are \emph{essentially countable}.
Both notions are well studied.
The former was studied by Hjorth \cite{HjorthBook}, the theory of \emph{turbulence} provides a dynamical obstacle for a continuous action of a Polish group to be classifiable by countable structures.
The latter stems from the work of Kechris \cite{KechrisLacunary}, is ultimately connected with the theory of \emph{countable Borel equivalnce relations} \cite{KechrisCBER}, and some recent progress shows that this notion is equivalent to a geometrical notion of $\sigma$-lacunarity \cite{GrebikLacunary} introduced and characterized by Miller in \cite{MillerLacunary}.

Our main contribution to the study of Borel equivalence relations induced by actions of tsi Polish groups is two-fold.
First, we characterize classification by countable structures with a variant of the $\I G_0$-dichotomy and provide a basis for non-classification by countable structures using a natural class of equivalence relations tightly connected to \emph{lsc submeasures} and \emph{$c_0$-equalities}, see \cite{Farah1,Farah2} and \cite[Chapters~3 and~15]{Kanovei}.
Second, we show that if such an equivalence relation admits classification by countable structures but is not essentially countable, then there is a Borel reduction from a canonical such equivalence relation, $\I E_0 ^{\I N}$.
Next, we formulate our results and describe high-level ideas.


\subsection{High-level overview of the arguments}

Recall that we denote as $E^X_G$ the orbit equivalence relation that is induced by a Borel action $G\curvearrowright X$ of a Polish group $G$ on a standard Borel space $X$, that is,
$$(x,y)\in E^X_G \ \Leftrightarrow \ \exists g\in G \ g\cdot x=y$$
for every $x,y\in X$.
The main technical tool that we use in this paper are variants of the \emph{$\I G_0$-dichotomy} of Kechris, Solecki and Todorcevic \cite{KST}.
Recall that the $\I G_0$-dichotomy characterizes Borel graphs of Borel chromatic number at most $\aleph_0$.
That is, a Borel graph either admits a decomposition into at most countably many independent sets or a Borel homomorphism from the graph $\I G_0$, a canonical example of graph that does not admit such a decomposition.
The elementary proof of this dichotomy is due to Miller \cite{MillerElementaryG0}.

Our approach is profoundly influenced by \cite{KechrisLacunary,HjorthDichotomy,HjorthBook,MillerLacunary}.
In particular, the results of Miller \cite{MillerLacunary} are not only literally used as a part of the proof of our main result but the ideas are internally present throughout the paper.

The starting goal of this project, also suggested by Miller, was to understand Hjorth's $\I E_2$-dichotomy \cite{HjorthDichotomy}.
Recall that the equivalence relation $\I E_2$ is induced by the canonical action of the Banach space $\ell_1$ on $\I R^{\I N}$ and that this action is turbulent, i.e., $\I E_2$ does not admit classification by countable structures.
Hjorth's $\I E_2$-dichotomy states that a restriction of $\I E_2$ to any Borel subset of $\I R^{\I N}$ is either essentially countable or Borel bi-reducible with $\I E_2$.
To prove this result one can use a variant of the $\I G_0$-dichotomy as follows.
The dichotomy is applied to a family of oriented hypergraphs $\fH_{V,W}$, where for every open neighborhoods $W\subseteq V$ of the identity in $G$ we define an oriented hypergraph $\fH_{V,W}$ on $X$ by declaring $x\in X^{<\aleph_0}$ to be an edge if the consecutive elements are connected by elements in $W$ and the first and last element of $x$ are not connected by any element of $V$.
Fixing $V$ and applying a variant of the $\I G_0$-dichotomy to a decreasing sequence of open neighborhoods of $1_G$ gives either a homomorphism from a canonical $\I G_0$-like object or a decomposition into $\fH_{V,W}$-invariant sets.
In the case of homomorphism one can use a \emph{refinement} technique, see \cite{MillerLacunary}, to find a reduction from $\I E_2$.
In the other case, it is not hard to see that the corresponding decomposition is made of sets that intersect each orbit in bounded and separated islands (or galaxies, aka grainy sets \cite[Chapter~15.2]{Kanovei}).
Formally, we make the following definition, see \cref{sec:PropIC}.
\begin{definition}[Property (IC)]
We say that $E^X_G$ satisfies \emph{Property (IC)} if for every open neighborhood $V$ of $1_G$ there is a sequence of Borel sets $(A_{l})_{l\in\I N}$ such that
\begin{itemize}
    \item for every $l\in \I N$ there is an open neighborhood $W_l$ of $1_G$ such that $A_{l}$ is $\fH_{V,W_l}$-independent,
    \item $X=\bigcup_{l\in \I N} A_{l}$.
\end{itemize}
\end{definition}
To understand why we derive essential countability in this case, it might be illustrative to make a detour and discuss Kechris' result \cite{KechrisLacunary} that says that every equivalence relation induced by a Borel action of a \emph{locally compact} Polish group is essentially countable.
As a first step in the original proof Kechris basically shows that every such action satisfies Property (IC).
The reason why we can conclude in both examples that the equivalence relation is essentially countable is a combination of Property (IC) and the following notion.
We say that an action $G\curvearrowright X$ admits a \emph{$G$-bounded topology} if there is a compatible Polish topology $\tau$ on $X$ that makes the action continuous\footnote{Compatible topology that makes the action continuous is called a \emph{$G$-Polish topology}.} and such that for every vertex $x\in X$ there is an open neighborhood $\Delta$ of $1_G$ such that acting with group elements from $\Delta$ on $x$ does not approximate elements from different orbit, i.e., the $\tau$-closure of $\Delta\cdot x$ is a subset of the orbit of $x$.
It is easy to see that if $G$ is compact and the action is continuous, then the topology is $G$-bounded.
Also the canonical topology on $\I R^{\I N}$ is $\ell_1$-bounded for the canonical action of the Banach space $\ell_1$.
The proof of Hjorth's theorem and Kechris' theorem is then finished by our first result
\begin{theorem}\label{thm:PropICandGbounded}
Suppose that a Borel equivalence relation $E^X_G$ satisfies Property (IC).
Then $E^X_G$ is essentially countable if and only if it admits a $G$-bounded topology. 
\end{theorem}

A natural question is to understand in which situation we can apply the aforementioned variant of the $\I G_0$-dichotomy for the oriented hypergraphs $\fH_{V,W}$ and get similar results.
It is not difficult to show that if we have a decomposition into $\fH_{V,W}$-independent sets, then the action is not turbulent (this result hold for any Polish group).
We obtain the following result.

\begin{theorem}\label{thm:CCS=IC}
Suppose that $E^X_G$ is a Borel equivalence relation induced by a Borel action of a \emph{tsi Polish group} $G$, that is, $G$ admits a two-sided invariant compatible metric.
Then the following are equivalent
\begin{itemize}
	\item $X$ satisfies Property (IC),
	\item $E^X_G$ admits classification by countable structures.
\end{itemize}
\end{theorem}

The reason why Property (IC) implies classification by countable structures is intuitively clear, the corresponding Borel decomposition into $\fH_{V,W}$-independent sets mimics the behavior of actions of \emph{non-archimedean} groups.
For this class Property (IC) holds trivially, and it is well-known that the corresponding orbit equivalence relations admit classification by countable structures.
The other implication is more challenging.
By a variant of the $\I G_0$-dichotomy we get a homomorphism from some canonical object.
We use heavily the assumption that the group is tsi to do several refinements.
Ultimately we obtain a Borel reduction from a Borel equivalence relation that is closely connected to \emph{tall lsc submeasures}, or \emph{$c_0$-equalities}, \cite{Farah1,Farah2,Kanovei}.
These are known to be induced by turbulent actions and this can be used to show that $E^X_G$ cannot be classifiable by countable structures.
Even though we are not able not show that \emph{tall lsc submeasures} and \emph{$c_0$-equalities} form a base for non-classification by countable structures, our result serves as an indication that this might be the case for equivalence relations that are induced by actions of tsi Polish groups.
For general Borel equivalence relations, in particular, the ones induced by actions of general Polish groups, we do not have any intuition.

Next we turn our attention to our main result.
We use results of Miller \cite{MillerLacunary} to show that $\I E_0^{\I N}$ is a canonical obstacle for essential countability under the assumption that $E^X_G$ admits classification by countable structures.
In the case of \emph{non-archimedean} tsi Polish groups, the result was proved by Hjorth and Kechris \cite{HjorthKechris}.
Recently Miller found a proof that uses a variant of the $\I G_0$-dichotomy \cite{MillerLacunary}.
We manage to weaken the assumption to merely tsi Polish groups but we need to keep the assumption that the equivalence relation admits classification by countable structures.
Note that without this assumption the situation is more complicated, e.g., $\I E_2$ is not essentially countable but there is no reduction from $\I E_0^{\I N}$ to $\I E_2$.
In general, there is no reduction from $\I E_0^{\I N}$ if the corresponding action admits $G$-bounded topology.

\begin{theorem}\label{thm:E_3dichotomy}
Suppose that $E^X_G$ is a Borel equivalence relation that admits classification by countable structures and is induced by a Borel action of a \emph{tsi Polish group} $G$, that is, $G$ admits a two-sided invariant compatible metric.
Then the following are equivalent:
\begin{itemize}
	\item $E^X_G$ is essentially countable,
	\item the action admits $G$-bounded topology,
	\item $\I E_0^{\I N}\not\le_B E^X_G$.
\end{itemize}
\end{theorem}

The strategy for proving this result is to use the $\I G_0$-dichotomy \emph{two times}.
First, by \cref{thm:CCS=IC} we get that Property (IC) holds, i.e., there is a Borel decomposition that mimics the behavior of non-archimedean groups.
Second, we use the result of Miller \cite{MillerLacunary} who found a variant of the $\I G_0$-dichotomy that characterizes $\sigma$-lacunarity (a formal strengthening of essential countability).
If we get a Borel decomposition into independent sets in his result, then we conclude that the action is $\sigma$-lacunary, i.e., essentially countable.
In the other case we get a homomorphism from some canonical object.
Miller was able to refine the homomorphism to get a reduction from $\I E_0^{\I N}$ under the assumption that the group is non-archimedean.
We show that his argument goes through under the weaker assumption of Property (IC).

\section*{Acknowledgement}
The author is indebted to Ben D. Miller for introducing him into the topic during his AKTION stay at KGRC in Vienna in 2017.
Also he would like to thank Ben D. Miller and Zoltán Vidnyászky for many engaging discussions and Ilijas Farah for useful suggestions. 
The research was supported by Leverhulme Research Project Grant RPG-2018-424 and by the GACR project 17-33849L and RVO: 67985840.

\section{Preliminaries}

\newcommand{\diag}{\operatorname{diag}}

For a set $X$ we write $X^{< \I N}$ for the set of all finite sequences of $X$.
Let $x\in X^{< \I N}$.
We define $|x|\in \I N$ to be the length of $x$ and write $x_i$ for the $i$-th element of $x$ for every $i<|x|$.
That is $x_0$ is the first element and $x_{|x|-1}$ is the last element of $x$ in this notation.
We set
$$\diag_X=\left\{x\in X^{<\I N}:\exists i\not= j< |x| \ x_i=x_j\right\}\cup X\cup \{\emptyset\}.$$
A \emph{relation $R$} on $X$ is any subset of $X^{<\I N}$.
A relation $\fH$ is a \emph{(finite-dimensional) dihypergraph} on $X$ if $\fH\cap \diag_X=\emptyset$ and it is a \emph{digraph} if $\fH\subseteq X^2$.
If $\mathcal{H}$ is a dihypergraph (or digraph) on $X$ and $A\subseteq X$, then we say that $A$ is $\mathcal{H}$-independent if the \emph{restriction} of $\fH$ to $A$, in symbols
$$\fH \upharpoonright A=\mathcal{H}\cap A^{<\I N},$$
is empty.

Let $X,Y$ be sets and $\varphi:X\to Y$ be a map.
The coordinate-wise extension $\overline{\varphi}$ of $\varphi$ to $X^{<\I N}$ is defined as
$$\overline{\varphi}(x)_i=\varphi(x_i)$$
for every $x\in X^{<\I N}$ and $i<|x|$.
We abuse the notation and write $\varphi$ instead of $\overline{\varphi}$.
Suppose that we have collections $(R_j)_{j\in I}$ and $(S_j)_{j\in I}$ of relations on $X$ and $Y$, respectively, where $I$ is some index set.
We say that a map $\varphi:X\to Y$ is a {\it homomorphism from $(R_j)_{j\in I}$ to $(S_j)_{j\in I}$} if 
$$x\in R_j \ \Rightarrow \ \varphi(x)\in S_j$$
for every $x\in X^{<\I N}$ and $j\in I$.
Moreover, it is a \emph{reduction} if we have
$$x\in R_j \ \Leftrightarrow \ \varphi(x)\in S_j$$
for every $x\in X^{<\I N}$ and $j\in I$.

\subsection{Polish $G$-space}

A topological space $X$ is a \emph{Polish space} if the underlying topology is separable and completely metrizable.
A \emph{standard Borel space} $X$ is a set endowed with a $\sigma$-algebra that is a $\sigma$-algebra of Borel sets for some Polish topology on $X$.
We call such a Polish topology \emph{compatible}.
A topological group $G$ is a \emph{Polish} group if the underlying topology is Polish.
We denote the \emph{$\sigma$-ideal of meager sets} on $G$ as $\mathcal{M}_G$.
We use the category quantifiers $\exists^*$, $\forall^*$ in the standard meaning, e.g.,
$$\forall^* g\in U \ P(g) \ \Leftrightarrow \ \{g\in U: \neg P(g)\}\in \mathcal{M}_G$$
$$\exists^* g\in U \ P(g) \ \Leftrightarrow \ \{g\in U: P(g)\}\not\in \mathcal{M}_G$$
where $U\subseteq G$ is an open set and $P$ is some property, see \cite{Gao}.

A \emph{Borel action} $G\curvearrowright X$ of a Polish group $G$ on a standard Borel space $X$ is an action that is additionally Borel measurable as a function from $G\times X$ to $X$.
We write $(g,x)\mapsto g\cdot x$ for the evaluation of the action at particular elements $g\in G$ and $x\in X$.
Similarly, we define $V\cdot x$ for any $V\subseteq G$ and $x\in X$.
We denote as $E^X_G$ the induced equivalence relation and as $[x]_{E^X_G}$ the equivalence class, or orbit, of $x\in X$.
A set $A\subseteq X$ is \emph{$G$-invariant} if it is a union of equivalence classes of $E^X_G$.
If $V\subseteq G$, then we set $(x,y)\in R^X_V$ if and only if $y\in V\cdot x$.
It is a result of Becker and Kechris \cite{BeckerKechris} that one can always find a compatible Polish topology on $X$ such that the action is continuous.
Any such Polish topology is called a \emph{$G$-Polish topology}.
If such a topology is fixed we say that $G\curvearrowright X$ is a \emph{Polish $G$-space}.
For $x\in X$ and $A\subseteq X$ we set $G(x,A)=\{g\in G:g\cdot x\in A\}$.

\begin{definition}
Let $X$ be a Polish $G$-space.
We say that $C\subseteq X$ is a \emph{$G$-lg comeager set} if $G\setminus G(x,C)\in \mathcal{M}_G$ for every $x\in X$.
Equivalently,
$$\forall^*g\in G \ g\cdot x\in C$$
holds for every $x\in X$
\end{definition}

\subsection{Borel reducibility}

A \emph{Borel equivalence relation} $E$ on a standard Borel space $X$ is is an equivalence relation that is additionally a Borel subset of $X\times X$.
We assume throughout the paper that the orbit equivalence relations of the form $E^X_G$ that we consider are always Borel equivalence relations.
A Borel equivalence relation $E$ on $X$ is \emph{Borel reducible} to a Borel equivalence $F$ on $Y$, in symbols $E\le_B F$, if there is a Borel map $\varphi:X\to Y$ that is a reduction from $E$ to $F$.

We say that a Borel equivalence relation $E$ is \emph{essentially countable} if it is Borel reducible to some \emph{countable Borel equivalence relation}, see \cite{KechrisCBER}.\footnote{A Borel equivalence relation is countable if the cardinality of each equivalence class is at most countable.}
In our setting, that is Borel equivalence relations inudeced by Polish group actions, this is equivalent to $\sigma$-lacunarity \cite{GrebikLacunary}.

We say that a Borel equivalence relation $E$ is \emph{classifiable by countable structures} if it is Borel reducible to some equivalence relation induced by a Borel action of $S_\infty$, the permutation group of $\I N$.
There are several other equivalent characterizations of classification by countable structures \cite[Theorem~12.3.3]{Kanovei} or \cite{Gao}.

The benchmark examples of Borel equivalence relations that we consider in this paper are $\I E_0$, $\I E_2$ and $\I E_0^{\I N}$.
The equivalence relation $\I E_0$ on $2^{\I N}$ is defined as $(x,y)\in \I E_0$ if and only if $|\{n\in \I N:x(n)\not=y(n)\}|<\aleph_0$.
The equivalence relation $\I E_2$ is defined in the Introduction.
The equivalence relation $\I E_0^{\I N}$ is the countable product of $\I E_0$, that is, $\I E_0^{\I N}$ is the equivalence relation on $2^{\I N\times\I N}$ defined as $(x,y)\in \I E_0^{\I N}$ if and only if $|\{m\in \I N:x(n,m)\not=y(n,m)\}|<\aleph_0$ holds for every $n\in \I N$.
We refer the reader to \cite{Kanovei} for more information about these particular examples.

\subsection{Turbulence}

Let $X$ be a Polish $G$-space, $V\subseteq G$, $U\subseteq X$ and $x\in X$.
We introduce some notation that is connected to the definition of local orbit, see~\cite[Section~10.2]{Gao}.
First, we define
$$\fJ(V)=\{x\in X^{<\I N}\setminus \diag_X:\forall i<|x|-1 \ (x_i,x_{i+1})\in R^X_V\},$$
the set of all $V$-jumps.
Let $\fJ(x,V)=\left\{y\in \fJ(V):y_0=x\right\}$.
Assuming now that $U$ and $V$ are open neighborhoods of $x\in X$ and the identity $1_G\in G$, respectively, we define the local orbit 
$$\fO(x,U,V)=\left\{y_{|y|-1}\in U:y\in \fJ(x,V)\cap U^{<\I N}\right\}.$$
That is, $\fO(x,U,V)$ are those elements of $U$ that are reachable from $x$ by $V$-jumps within $U$.

\begin{definition}[Section~10 \cite{Gao}]
Let $G$ be a Polish group and $X$ be a Polish $G$-space.
We say that the action $G\curvearrowright X$ is \emph{turbulent} if
\begin{itemize}
    \item every equivalence class of $E^X_G$ is dense and meager in $X$,
    \item the local orbit $\fO(x,U,V)$ is somewhere dense for every $x\in X$ and every open neighborhoods $U$ and $V$ of $x\in X$ and $1_G\in G$, respectively. 
\end{itemize}
\end{definition}

Let $E$ be an equivalence relation on a Polish space $X$.
We say that $E$ is \emph{generically $S_\infty$-ergodic} if for every Polish $S_\infty$-space $Y$ and every Baire measurable homomorphism $\varphi: X \to Y$ from $E$ to $E^Y_{S_\infty}$ there is $y\in Y$ such that $\varphi^{-1}\left([y]_{E^Y_{S_\infty}}\right)$ is comeager in $X$.

\begin{theorem}[Corollary 10.4.3 \cite{Gao}]\label{thm:GenErg}
Let $G$ be a Polish group and $X$ be a Polish $G$-space.
Suppose that the action $G\curvearrowright X$ is turbulent.
Then $E^X_{G}$ is generically $S_\infty$-ergodic.
In particular, $E^X_G$ is not classifiable by countable structures.
\end{theorem}

\subsection{Tsi Polish groups}

A Polish group $G$ is \emph{tsi} (states for two-sided invariant) if there is an open basis at $1_G$ made of conjugacy invariant open sets.
That is, there is a sequence $(\Delta_k)_{k\in \I N}$ of open neighborhoods of $1_G$ that is an open base and such that $g\cdot \Delta_k \cdot g^{-1}=\Delta_k$ for every $g\in G$ and $k\in \I N$. 
Equivalently, see~\cite[Exercise~2.1.4]{Gao}, there is a compatible metric $d$ on $G$ that is two sided invariant, i.e., $d(g,h)=d(h^{-1}\cdot g,1_G)=d(g\cdot h^{-1},1_G)$ for every $g,h\in G$.
It follows from \cite[Exercise~2.2.4]{Gao} that such a metric $d$ is necessarily complete.
We always assume that such a metric $d$ on $G$ is fixed and put $\Delta_\epsilon=\{g\in G:d(g,1_G)<\epsilon\}$.
We have $g\cdot \Delta_\epsilon \cdot g^{-1}=\Delta_\epsilon$ for every $\epsilon>0$ and $g\in G$.
We abuse the notation and define $\Delta_k=\Delta_{\frac{1}{2^k}}$ for every $k\in \I N$.
Note that $(\Delta_k)_{k\in \I N}$ is a conjugacy invariant open base at $1_G$ such that $\Delta_{k+1}\cdot \Delta_{k+1}\subseteq \Delta_k$ and $\Delta_k=\Delta_k^{-1}$ for every $k\in \I N$.

We define the dihypergraphs that we use in this paper.

\begin{definition}\label{def:dihyp}
Let $X$ be a Polish $G$-space and $k,m\in \I N$.
We set
$$\fH_{k,m}=\left\{x\in X^{<\I N}:x\in \fJ(\Delta_m) \ \wedge \ (x_0,x_{|x|-1})\not\in R^X_{\Delta_k}\right\}.$$
We note that the definition makes sense for any Polish group $G$ and any sequence of (symmetric) neighborhoods of $1_G$.
\end{definition}

\subsection{$\I G_0$-dichotomy}\label{subsec:G_0}

We formulate three versions of the $\I G_0$-dichotomy.
First version is the original dichotomy and the other two are the versions that we use in this paper.
We formulate the statements in bigger generality as it is done in \cite{MillerLectureNotes}.

For $s\in 2^{<\I N}$ define the graph
$$\I G_s=\{(s^\frown (0)^\frown c,s^\frown (1)^\frown c):c\in 2^\I N\}$$
on $2^\I N$.
Fix some dense collection $(s_n)_{n\in \I N}\subseteq 2^{<\I N}$ such that $|s_n|=n$, i.e., $s_n\in 2^n$, for every $n\in \I N$.
Here \emph{dense} means that for every $s\in 2^{<\I N}$ there is $n\in \I N$ such that $s\sqsubseteq s_n$.
Set $\I G_0=\bigcup_{n\in \I N}\I G_{s_n}$.

\begin{theorem}[$\I G_0$-dichotomy \cite{KST}]
Suppose that $X$ is a Hausdorff space and $G$ is an analytic graph on $X$.
Then exactly one of the following holds:
\begin{enumerate}
    \item there is a sequence $(B_k)_{k\in \I N}$ of Borel subsets of $X$ such that $X=\bigcup_{k\in \I N}B_k$ and $B_k$ is $G$-independent for every $k\in \I N$,
    \item there is a continuous homomorphism $\varphi:2^{\I N}\to X$ from $\I G_0$ to $G$.
\end{enumerate}
\end{theorem}

Next we formulate two versions of the $\I G_0$-dichotomy that we use in this paper.
For the first one we need to recall some notation from \cite{MillerLacunary}.
Let $k_n\in \I N$ be such that $k_0=0$, $k_{n+1}\le \max\{k_m:m\le n\}+1$ for every $n\in\I N$ and for every $k\in\I N$ there are infinitely many $n\in\I N$ such that $k_n=k$.
Fix $(s_n)_{n\in \I N}\in 2^{\I N}$ such that $|s_n|=n$ for every $n\in \I N$ and $\{s_n:k_n=k\}$ is dense in $2^{<\I N}$ for every $k\in \I N$.
Set
$$\I G_{0,k}=\bigcup \left\{\I G_{s_n}:k_n=k\right\}$$
for every $k\in \I N$.

\begin{theorem}[Theorem~2 \cite{MillerLacunary}]\label{thm:G_0Ben}
Suppose that $X$ is a Hausdorff space and $(G_{i,j})_{i,j\in \I N}$ is an
increasing-in-$j$ sequence of analytic digraphs on $X$.
Then exactly one of the following holds:
\begin{enumerate}
    \item there is a sequence $(B_i)_{i\in \I N}$ of Borel subsets of $X$ such that $X=\bigcup_{i\in \I N}B_i$ and
    the Borel chromatic number of $G_{i,j}\upharpoonright B_i$ is at most countable for every $i,j\in \I N$,
    \item there exist a function $f:\I N\to \I N$ and a continuous homomorphism $\varphi:2^{\I N}\to X$ from $(\I G_{0,k})_{k\in \I N}$ to $(G_{k,f(k)})_{k\in \I N}$.
\end{enumerate}
\end{theorem}

We define a generalization of $\I G_0$ and $\I E_0$ for special class of finitely branching trees.
We say that a tree $T\subseteq \I N^{<\I N}$ is {\it finitely uniformly branching} if there is a sequence $(\ell^T_m)_{m\in \I N}$ of natural numbers such that $\ell^T_m\ge 2$ for every $m\in \I N$ and we have
$$\ell^T_{|s|}=\{i\in \I N:s^\frown (i)\in T\}$$
for every $s\in T$.
If $T$ is a tree and $s\in T$, then we define $T_s=\{t\in \I N^{<\I N}:s^\frown t\in T\}$.
Note that $T_s=T_t$ whenever $t,s\in T$ and $|t|=|s|$.
We denote as $[T]\subseteq \I N^{\I N}$ the set of all branches through $T$, i.e., $\alpha\in [T]$ if and only if $\alpha\upharpoonright m\in T$ for every $m\in \I N$.

\begin{definition}
Let $T$ be a finitely uniformly branching tree and $s\in T$.
The dihypergraph $\I G^T_s$ on $[T]$ is defined as
$$\I G^T_s=\left\{(s^\frown (i)^\frown \alpha)_{i<\ell^T_{|s|}}:\alpha\in [T_{s^\frown (0)}]\right\}.$$
\end{definition}

The equivalence relation $\I E^T_0$ on $[T]$ is defined as
$$(\alpha,\beta)\in \I E^T_0 \ \Leftrightarrow \ |\{n\in \I N:\alpha(n)\not=\beta(n)\}|<\aleph_0$$
where $\alpha,\beta\in [T]$.
Note that in the case when $T=2^{<\I N}$ we have $\I E^T_0=\I E_0$.

\begin{theorem}[$\I G_0$-dichotomy for dihypergraphs, Theorem 2.2.12 \cite{MillerLectureNotes}]\label{thm:HypG0}
Let $X$ be a Hausdorff space and let $(\fH_m)_{m\in \I N}$ be a sequence of analytic dihypergraphs on $X$. Then at least one of the following holds:
\begin{enumerate}
    \item there is a sequence $(B_k)_{k\in \I N}$ of Borel subsets of $X$ such that $X=\bigcup_{k\in \I N}B_k$ and for every $k\in \I N$ there is $m(k) \in \I N$ such that $B_k$ is $\fH_{m(k)}$-independent,
    \item  there is a finitely uniformly branching tree $T$, a dense sequence $(s_m)_{m\in \I N}\subseteq T$ such that $|s_m| = m$ for every $m\in \I N$ and a continuous homomorphism $\varphi:[T]\to X$ from $(\I G^T_{s_m})_{m\in \I N}$ to $(\fH_m)_{m\in \I N}$.
\end{enumerate}
Moreover, if the sequence $(\fH_m)_{m\in \I N}$ is decreasing then the conditions are mutually exclusive.
\end{theorem}

\subsection{Borel pseudometrics}\label{sec:pseuodmetric}

\begin{definition}
Let $T$ be a finitely uniformly branching tree.
A function ${\bf d}:[T]\times [T]\to [0,+\infty]$ is called a \emph{Borel pseudometric} if
\begin{enumerate}
	\item ${\bf d}$ is pseudometric,
	\item ${\bf d}^{-1}([0,\epsilon))$ is a Borel subset of $[T]\times [T]$ for every $\epsilon>0$,
	\item $(\{\beta:{\bf d}(\alpha,\beta)<+\infty\},{\bf d})$ is a separable pseudometric space for every $\alpha\in [T]$,
	\item if $\alpha_n\to_{[T]} \alpha$ and $(\alpha_n)_{n\in \I N}$ is a ${\bf d}$-Cauchy sequence, then ${\bf d}(\alpha_n,\alpha)\to 0$.
\end{enumerate}
Moreover, we say that a Borel pseudoemtric is \emph{uniform} if
\begin{itemize}
	\item for every $m\in \I N$, $s,t\in T\cap \I N^{m}$ and $\alpha,\beta\in [T_s]=[T_t]$ we have 
	$$\left|{\bf d}(s^\frown \alpha,t^\frown \alpha)-{\bf d}(s^\frown \beta,t^\frown \beta)\right|<\frac{1}{2^m},$$
	$$\left|{\bf d}(s^\frown \alpha,s^\frown \beta)-{\bf d}(t^\frown \alpha,t^\frown \beta)\right|<\frac{1}{2^m}$$
	where we set $|(+\infty)-(+\infty)|=0$.
\end{itemize}
\end{definition}

We describe a canonical way how to find Borel pseudometrics.
Recall that if $G$ is a tsi Polish group, then $d$ is a fixed compatible two-sided invariant metric on $G$.

\begin{proposition}\label{pr:Borel pseudo}
Let $G$ be a tsi Polish group, $X$ be a Polish $G$-space such that $E^X_G$ is Borel, $T$ be a finitely uniformly branching tree and $\varphi:[T]\to X$ be a continuous map.
Then the function ${\bf d}:[T]\times [T]\to [0,+\infty]$ defined as
$${\bf d}(\alpha,\beta)=\inf\{d(g,1_G):g\in G \ \wedge \ g\cdot \varphi(\alpha)=\varphi(\beta)\}$$
for $\alpha,\beta\in [T]$ is a Borel pseudometric. 
\end{proposition}
\begin{proof}
(1.)
The invariance of $d$ guarantees that $d(g,1_G)=d(g^{-1},1_G)$ for every $g\in G$ and consequently that ${\bf d}$ is symmetric.
Let $\alpha,\beta,\gamma\in [T]$.
We may assume that ${\bf d}(\alpha,\beta)+{\bf d}(\beta,\gamma)<+\infty$.
In that case for every $\epsilon>0$ there is $g,h\in G$ such that $d(g,1_G)<{\bf d}(\alpha,\beta)+\epsilon$ and $d(h,1_G)<{\bf d}(\beta,\gamma)+\epsilon$.
Then we have 
$${\bf d}(\alpha,\gamma)-2\epsilon\le d(h\cdot g,1_G)-2\epsilon\le d(h,1_G)+d(g,1_G)-2\epsilon<{\bf d}(\alpha,\beta)+{\bf d}(\beta,\gamma)$$
because $d(h\cdot g,1_G)\le d(h\cdot g,g)+d(g,1_G)=d(h,1_G)+d(g,1_G)$ by the invariance of $d$.

(2.)
Recall that for $\epsilon>0$ we defined $\Delta_{\epsilon}=\{g\in G:d(g,1_G)<\epsilon\}$.
It follows from our assumption that $E^X_G$ is a Borel equivalence relation together with \cite[Theorem~7.1.2]{Gao} that the relation $R^X_{\Delta_{\epsilon}}$ is Borel for every $\epsilon>0$.
We have
$${\bf d}^{-1}([0,\epsilon))=\left\{(\alpha,\beta)\in [T]\times [T]:{\bf d}(\alpha,\beta)<\epsilon \right\}=\left(\varphi^{-1}\times \varphi^{-1}\right)(R^X_{\Delta_\epsilon})$$
and that shows (2).

(3.)
Let $\alpha\in [T]$.
The space $G_\alpha=\{g\in G:\exists \beta\in [T] \ g\cdot\varphi(\alpha)=\varphi(\beta)\}$ endowed with $d$ is a separable metric space.
It is easy to see that the assignment $g\mapsto \beta$ where $g\cdot \varphi(\alpha)=\varphi(\beta)$ is a contraction from $(G_{\alpha},d)$ to the quotient metric space $(\{\beta:{\bf d}(\alpha,\beta)<+\infty\}/{\bf d},\bf d)$.

(4.)
Let $(\alpha_n)_{n\in \I N},\alpha\in [T]$ satisfy the assumptions of (4).
After possibly passing to a subsequence we may suppose that there is a sequence $(g_n)_{n\in \I N}\subseteq G$ such that $g_n\cdot \varphi(\alpha_n)=\varphi(\alpha_{n+1})$ and $d(g_n,1_G)<\frac{1}{2^n}$.
Define $h^n_m=g_{n-1}\cdot {\dots} \cdot g_m$ for every $m<n\in \I N$.
Then it follows that $(h^n_m)_{n\in \I N}$ is $d$-Cauchy whenever $m\in \I N$ is fixed.
Since $d$ is complete there is $(h_m)_{m\in \I N}\in G$ such that $h^n_m\to h_m$ for every $m\in \I N$.
Moreover, we have $d(h_m,1_G)<\frac{1}{2^{m-1}}$.
Continuity of the action and of the map $\varphi$ gives
$$h_m\cdot \varphi(\alpha_m)\leftarrow h^n_m\cdot \varphi(\alpha_m)=\varphi(\alpha_n)\to \varphi(\alpha).$$
This finishes the proof.
\end{proof}

Every Borel pseudoemtric ${\bf d}$ on $[T]$ defines a Borel equivalence relation $F_{\bf d}$ on $[T]$ as
$$(\alpha,\beta)\in F_{\bf d} \ \Leftrightarrow \ {\bf d}(\alpha,\beta)<+\infty.$$
Note that in the case of \cref{pr:Borel pseudo} we have that $F_{{\bf d}}=\left(\varphi^{-1}\times \varphi^{-1}\right)(E^X_G)$.

\begin{theorem}\label{th:uniform nonmeager is everything}
Let $T$ be a finitely uniformly branching tree and ${\bf d}$ be a uniform Borel pseudometric such that $\I E^T_0\subseteq F_{\bf d}$.
Then the following are equivalent
\begin{itemize}
	\item [(a)] $F_{\bf d}$ is nonmeager,
	\item [(b)] $F_{\bf d}=[T]\times [T]$.
\end{itemize}
\end{theorem}
\begin{proof}
(b) $\Rightarrow$ (a) is trivial.
We show that (a) $\Rightarrow$ (b).
Suppose first, that for every $k\in \I N\setminus \{0\}$ there is $m_k\in \I N$ such that ${\bf d}(\alpha,\beta)<\frac{1}{k}$ for every $\alpha,\beta\in [T]$ such that $\left\{n\in \I N:\alpha(n)\not=\beta(n)\right\}\cap m_k=\emptyset$ and $(\alpha,\beta)\in \I E^T_0$.
We may assume that $(m_k)_{k\in \I N}\subseteq \I N$ is strictly increasing and we set $m_0=0$.

Let $x,y\in [T]$ and define $y_k\in [T]$ such that $y_k\upharpoonright m_k=y$ and $y_k(n)=x(n)$ for every $n\ge m_k$.
Then clearly $y_0=x$, $(y_r,y_s)\in \I E^T_0\subseteq F_{\bf d}$ for every $r,s\in \I N$ and $y_k\to_{[T]} y$.
Let $k\in \I N\setminus \{0\}$ and $r,s\ge k$.
Then we have 
$$|\{n\in \I N:y_r(n)\not =y_s(n)\}|\cap m_k=\emptyset$$
and consequently ${\bf d}(y_r,y_s)<\frac{1}{k}$.
This shows that $(y_k)_{k\in \I N}$ is a ${\bf d}$-Cauchy sequence and by (4) from the definition of Borel pseudometric we have ${\bf d}(y_k,y)\to 0$.
In particular, there is $k\in \I N$ such that ${\bf d}(y_k,y)<+\infty$ and therefore $(y_k,y)\in F_{\bf d}$.
Altogether we have $(x,y)\in F_{\bf d}$ and since $x,y\in [T]$ were arbitrary we have that $F_{\bf d}=[T]\times [T]$.

Second, suppose that there is $\epsilon>0$ such that for every $m\in \I N$ there are $\alpha_m,\beta_m\in [T]$ such that ${\bf d}(\alpha,\beta)>\epsilon$, $\{n\in \I N:\alpha(n)\not=\beta(n)\}\cap m=\emptyset$ and $(\alpha_m,\beta_m)\in \I E^T_0$.
We show that this contradicts $F_{\bf d}$ being non-meager.

Note that $F_{\bf d}$ is a Borel equivalence relation and every $F_{\bf d}$-equivalence class is dense in $[T]$ because $\I E^T_0\subseteq F_{\bf d}$.
This implies, by \cite[Theorem~8.41]{kechrisclassical}, that there is $\alpha\in [T]$ such that $[\alpha]_{F_{\bf d}}$ is comeager in $[T]$.
By (3.) in the definition of Borel pseudometric, there are Borel sets $(U_\ell)_{\ell\in \I N}$ such that $\bigcup_{\ell\in \I N} U_l=[\alpha]_{F_{\bf d}}$ and 
$${\bf d}(x,y)<\frac{\epsilon}{2}$$
for every $\ell\in \I N$ and $x,y\in U_\ell$.

By \cite[Proposition~8.26]{kechrisclassical}, we find $t'\in T$ and $\ell\in \I N$ such that $U_\ell$ is comeager in ${t'}^\frown [T_t]$.
Pick $m\in \I N$ such that $m\ge |t'|$ and $\frac{1}{m}<\frac{\epsilon}{4}$.
We may suppose that $\alpha_m=s^\frown {u_0}^\frown x$ and $\beta_m=s^\frown {u_1}^\frown x$ where $|s|= m$, $|u_0|=|u_1|$ and $x\in [T_{s^\frown u_0}]=[T_{s^\frown u_1}]$.

Let $t\in T$ be such that $t'\sqsubseteq t$ and $|t|=|s|=m$.
Then we have that $U_\ell$ is comeager in $t^\frown [T_t]$ and therefore there is $y\in [T_{t^\frown u_0}]=[T_{t^\frown u_1}]$ such that
$$t^\frown {u_0}^\frown y,t^\frown {u_1}^\frown y \in U_\ell.$$
In particular, we have ${\bf d}(t^\frown {u_0}^\frown y,t^\frown {u_1}^\frown y)<\frac{\epsilon}{2}$.

We use that ${\bf d}$ is uniform.
We have
$$\left|{\bf d}(s^\frown ({u_0}^\frown x),s^\frown ({u_1}^\frown x))-{\bf d}(t^\frown ({u_0}^\frown x),t^\frown ({u_1}^\frown x))\right|<\frac{1}{2^m}< \frac{1}{m}<\frac{\epsilon}{4}$$
and
$$\left|{\bf d}((t^\frown {u_0})^\frown x,(t^\frown {u_1})^\frown x)-((t^\frown {u_0})^\frown y,(t^\frown {u_1})^\frown y)\right|<\frac{1}{2^{|t^\frown u_0|}}< \frac{1}{m}<\frac{\epsilon}{4}.$$
This implies
$${\bf d}(t^\frown {u_0}^\frown y,t^\frown {u_1}^\frown y)\ge {\bf d}(s^\frown {u_0}^\frown x,s^\frown {u_1}^\frown x)-\frac{\epsilon}{2}>\frac{\epsilon}{2}$$
and that contradicts ${\bf d}(t^\frown {u_0}^\frown y,t^\frown {u_1}^\frown y)<\frac{\epsilon}{2}$.
This finishes the proof.
\end{proof}

\subsection{Base for non-classification by countable structures}\label{sec:BASE}

We describe the family of Borel equivalence relations that will serve as a base under Borel reducibility for non-classification by countable structures in the proof of \cref{thm:CCS=IC}.
To that end we recall definitions of two types of Borel equivalence relations that are well-studied \cite{Farah1,Farah2} and \cite[Chapters~3 and~15]{Kanovei}.
We denote the power set of $\I N$ as $\mathcal{P}(\I N)$.

A map $\Theta:\mathcal{P}(\I N)\to [0,+\infty]$ is a \emph{lsc submeasure} if $\Theta(\emptyset)=0$, $\Theta(M\cup N)\le \Theta(M)+\Theta(N)$ whenever $M, N\in \mathcal{P}(\I N)$, $\Theta(\{m\})<+\infty$ for every $m\in \I N$ and
$$\Theta(M)=\lim_{m\to\infty} \Theta(M\cap m)$$
for every $M\in \mathcal{P}(\I N)$.
We say that $\Theta$ is \emph{tall} if $\lim_{m\to \infty} \Theta(\{m\})=0$.

Let $\Theta$ be a tall lsc submeausre.
Then the equivalence relation $E_\Theta$ on $2^{\I N}$ is defined as
$$(x,y)\in E_\Theta \ \Leftrightarrow \ \lim_{m\to \infty}\Theta(\{n\in \I N\setminus m:x(n)\not=y(n)\})=0$$
for every $x,y\in 2^{\I N}$.
Similarly as in \cref{th:uniform nonmeager is everything}, we have that $E_\Theta$ is non-meager if and only if $E_\Theta=2^{\I N}\times 2^{\I N}$.
We refer the reader to \cite[Chapter~3]{Kanovei} for more information about lsc submeasures and their connection to Borel reducibility.

A sequence of finite metric spaces $((Z_m,\mathfrak{d}_m))_{m\in \I N}$ is called \emph{non-trivial} if 
$$\liminf_{m\to \infty} r(Z_m,\mathfrak{d}_m)>0 \ \& \ \lim_{m\to\infty} j(Z_m,\mathfrak{d}_m)=0$$
where $r(Z,\mathfrak{d})=\max \mathfrak{d}$ and $j(Z,\mathfrak{d})$ is the minimal $\epsilon>0$ such that there is $\ell\in \I N$ and a sequence $(z_0,\dots z_l)$ that contains every element of $Z$ and satisfies $\mathfrak{d}(z_i,z_{i+1})<\epsilon$ for every $i<l$.

Let $\mathcal{Z}=((Z_m,\mathfrak{d}_m))_{m\in \I N}$ be a non-trivial sequence of finite metric spaces and $\prod_{m\in \I N} Z_m$ be endowed with the product topology.
Then the equivalence relation $E_\mathcal{Z}$ on $\prod_{m\in \I N} Z_m$ is defined as
$$(x,y)\in E_\mathcal{Z} \ \Leftrightarrow \ \lim_{m\to \infty} \mathfrak{d}_m(x(m),y(m))=0$$
for every $x,y\in \prod_{m\in \I N} Z_m$.
We refer the reader to \cite{Farah1,Farah2,Kanovei} for more information about these equivalence relations and their connection to Borel reducibility.

\begin{definition}
Denote as $\mathfrak{B}$ the collection of all Borel meager equivalence relations that contain $E_\Theta$ for some tall lsc submeasure $\Theta$ or $E_\mathcal{Z}$ for some non-trivial sequence of finite metric spaces $\mathcal{Z}$.
That is for every $E\in \mathfrak{B}$ there is either tall lsc submeasure $\Theta$ such that $E_\Theta \subseteq E$ and $E$ is a meager subset of $2^{\I N}\times 2^{\I N}$, or there is a non-trivial sequence of finite metric spaces $\mathcal{Z}$ such that $E_\mathcal{Z}\subseteq E$ and $E$ is a meager subset of $\prod_{m\in \I N} Z_m \times \prod_{m\in \I N} Z_m$.
\end{definition}

\begin{theorem}\label{th:base for nonCCS}
Let $E\in \mathfrak{B}$.
Then $E$ is not classifiable by countable structures.
\end{theorem}
\begin{proof}
We start by recalling the following well-known facts, see~\cite[Appendix~3.7]{GrebikPhD} and~\cite[Chapter~16]{Kanovei}.
Suppose that $\Theta$ is a tall lsc submeasure and $E_\Theta$ is meager.
Then $E_{\Theta}$ is induced by a turbulent action of a Polish group on $2^\I N$. 
Similarly, if $\mathcal{Z}$ is a non-trivial sequence of finite metric spaces, then $E_{\mathcal{Z}}$ is induced by a turbulent action of a Polish group on $\prod_{m\in \I N} Z_m$.
In particlar, these equivalence relations do not admit classification by countable structures.

Let $E\in \mathfrak{B}$ be an equivalence relation on $Y$.
By the definition we find $F\subseteq E$ such that either $F=E_\Theta$ for some tall lsc submeasure $\Theta$ or $F=E_{\mathcal{Z}}$ for some non-trivial sequence of finite metric spaces $\mathcal{Z}$.
Now we use the fact that turbulent actions are generically $S_\infty$-ergodic.

Suppose for a contradiction that $E$ admits classification by countable structures.
That is, there is a Polish $S_\infty$-space $W$ and a Borel map $\psi:Y\to W$ that is a reduction from $E$ to $E^W_{S_\infty}$.
In particular, $\psi$ is a Borel homomorphism from $F$ to $E^W_{S_\infty}$.
It follows from \cref{thm:GenErg} that there is $y\in Y$ such that $\psi^{-1}([\psi(y)]_{E^W_{S_\infty}})$ is comeager in $Y$.
Since $\psi$ is a reduction we have
$$\psi^{-1}([\psi(y)]_{E^W_{S_\infty}})\subseteq [y]_E.$$
An application of \cite[Theorem~8.41]{kechrisclassical} shows that $E$ is comeager and that is a contradiction.
\end{proof}

\section{$G$-bounded Topology}

In this section we introduce the notion of $G$-bounded topology.
We show that this property is closed downwards in the Borel reducibility order and that essentially countable equivalence relation induced by group actions always admit such a topology.
Our main result is that if a Polish $G$-space admits a (finer) $G$-bounded topology, then $\I E_0^{\I N}$ is not Borel reducible to $E^X_G$.
Results in this section hold for all Polish groups, except for \cref{cor:HjorthKechris}.

Let $\tau$ be a topology on a space $X$ and $A\subseteq X$.
We write $\overline{A}^\tau$ for the $\tau$-closure of $A$ in $X$.
If $\tau$ is understood from the context, we omit the superscript.

\begin{definition}
Let $X$ be a Polish $G$-space and $\tau$ be the underlying $G$-Polish topology on $X$.
We say that $\tau$ is \emph{$G$-bounded} if
\begin{itemize}
	\item for every $x\in X$ there is an open neighborhood $\Delta$ of $1_G$ such that $\overline{\Delta\cdot x}^\tau\subseteq [x]_{E^X_G}$.
\end{itemize}
Instead of saying that $\tau$ is $G$-bounded $G$-Polish topology, we say simply \emph{bounded $G$-Polish topology}.
\end{definition}

We mentioned in the introduction that if $G$ is a locally compact Polish group, then any $G$-Polish topology on $X$ is $G$-bounded.
Similarly this holds when $G$ is a \emph{countable discrete group}.
Next we discuss the example with the Banach space $\ell_1$ in a greater detail.

\begin{example}
Let $\ell_1$ be the Banach space of all absolutely summable real sequences and consider the canonical action $\ell_1\curvearrowright \I R^{\I N}$ that is given by coordinate-wise summation.
Then the product topology turns $\I R^{\I N}$ into a Polish $\ell_1$-space.
Let $x,y\in \I R^{\I N}$ and $(a_n)_{n\in \I N}\subseteq \ell_1$ be such that $\|a_n\|_1\le 1$ for every $n\in \I N$ and $a_n\cdot x\to y$.
Fix $N\in \I N$, then we have
$$\sum_{k=0}^N |x(k)-y(k)|\le\sum_{k=0}^N |x(k)+a_n(k)-y(k)|+\sum_{k=0}^N |a_n(k)|\le$$
$$\le \sum_{k=0}^N |x(k)+a_n(k)-y(k)|+1\to 1$$
as $N\to \infty$.
This shows that $\|x-y\|_1\le 1$ and therefore we see that the product topology is $\ell_1$-bounded.
\end{example}

Next, we show that the existence of a $G$-bounded topology is closed downwards in the Borel reducibility order.
In the proof we use several technical but elementary results that are collected in \cref{app:technical}

\begin{theorem}\label{th:bounded downward closed}
Let $X$ be a Polish $G$-space and $Y$ be a Polish $H$-space such that $E^X_G$ is Borel, $E^Y_H\le_B E^X_G$ and the Polish topology on $X$ is $G$-bounded.
Then there is a finer $H$-Polish topology $\tau$ on $Y$ that is $H$-bounded.
\end{theorem}
\begin{proof}
Let $\varphi:Y \to X$ be a Borel reduction from $E^Y_H$ to $E^X_G$.
It is easy to see that the assumptions of \cref{lm:local continuity} and \cref{lm:global continuity} are satisfied.
Let $\tau$ be the finer Polish topology on $Y$ that is given by \cref{lm:global continuity} and $C\subseteq Y$ be a Borel $H$-lg comeager set that satisfies conclusions of both \cref{lm:local continuity} and \cref{lm:global continuity}.
We show that $\tau$ works as required.

Let $y\in C$ and $\Delta$ be an open neighborhood of $1_G$ such that 
$$\overline{\Delta\cdot \varphi(y)}\subseteq [\varphi(y)]_{E^X_G}.$$
\cref{lm:local continuity} then gives an open neighborhood $\Delta'$ of $1_H$ such that
$$\varphi(C\cap \Delta'\cdot y)\subseteq \Delta\cdot \varphi(y).$$
Let $z\in \overline{C\cap \Delta'\cdot y}^\tau\cap C$.
By the definition we find $y_n\in C\cap \Delta'\cdot y$ such that $y_n\to_\tau z$.
By \cref{lm:global continuity} we have that $\varphi(y_n)\to \varphi(z)$ because $z\in C$.
Note that $\varphi(y_n)\in \Delta\cdot \varphi(y)$.
This gives that $\varphi(z)\in [\varphi(y)]_{E^X_G}$.
Since $\varphi$ is a reduction we have $(y,z)\in E^Y_H$.
We see that $\tau$ and $C$ satisfies the assumption of \cref{lm:weak boundedness} and therefore $\tau$ is $H$-bounded.
\end{proof}

\begin{corollary}\label{cor:EC implies bounded}
Let $X$ be a Polish $G$-space such that $E^X_G$ is essentially countable.
Then there is a finer $G$-Polish topology $\tau$ on $X$ that is $G$-bounded.
\end{corollary}

Recall that $\I E_0^{\I N}$ is the countable product of $\I E_0$.
Since the latter is induced by the canonical continuous action $2^{<\I N}\curvearrowright 2^{\I N}$, it is not hard to see that the former is induced by the canonical continuous action of $(2^{<\I N})^{\I N}\curvearrowright 2^{\I N\times \I N}$.

\begin{theorem}\label{th:E_3 not bounded}
There is no finer bounded $\left(2^{<\I N}\right)^{\I N}$-Polish topology on $2^{\I N\times \I N}$.
\end{theorem}
\begin{proof}
First we introduce an auxiliary notation.
Set
$$G_k=\left\{\alpha\in (2^{<\I N})^{\I N}:\forall l<k \ \alpha(l)=1_{2^{<\I N}}\right\}$$
for every $k\in \I N$.
Then it is easy to see that $(H_k)_{k\in \I N}$ is an open basis of $1_{(2^{<\I N})^{\I N}}$ made of clopen subgroups.

Suppose that $\tau$ is a finer bounded $(2^{<\I N})^{\I N}$-Polish topology on $2^{\I N\times \I N}$.
Define
$$D_k=\left\{x\in 2^{\I N\times \I N}: \overline{G_k\cdot x}^\tau\subseteq [x]_{\I E_0^{\I N}}\right\}.$$
Note that $(D_k)_{k\in \I N}$ is an increasing sequence and we have $2^{\I N\times \I N}=\bigcup_{k\in \I N} D_k$.

\begin{claim}\label{cl:co-analytic}
The set $D_k$ is co-analytic for every $k\in \I N$.
\end{claim}
\begin{proof}
Fix an open basis $(U_r)_{r\in \I N}$ of $\tau$.
Then we have
$$x\in D_k \ \Leftrightarrow \ \forall \alpha\in G_k \ \forall y\in 2^{\I N\times \I N}
\left(\left(\exists r\in \I N \ y\in U_r \wedge \alpha\cdot x\not \in U_r\right)\vee (x,y)\in \I E_0^{\I N}\right).$$
The formula on the right-hand side is co-analytic because $\I E_0^{\I N}$ is a Borel relation.
\end{proof}

Using \cref{cl:co-analytic} and \cite[Proposition~8.26]{kechrisclassical} we find $k\in \I N$ and a basic open set $O\subseteq 2^{\I N\times \I N}$, in the (canonical) product topology, such that $D_k$ is comeager in $O$.
Note that since $D_k$'s are increasing we may assume that the first-coordinates of the indices that define $O$ are strictly less than $k$.

It follows from \cite[Theorem~8.38]{kechrisclassical} that there is a Borel set $C'\subseteq 2^{\I N\times \I N}$ that is comeager in the product topology such that the product topology and $\tau$ coincide on $C'$.
Define
$$C=\{x\in C':\forall^* \alpha\in (2^{<\I N})^{\I N}  \ \alpha\cdot x\in C'\}.$$
Then $C$ is a Borel set by \cite[Theorem~16.1]{kechrisclassical} and a routine use of \cite[Theorem~8.41]{kechrisclassical} shows that $C$ is comeager in the product topology.
Then clearly $C\subseteq C'$ and therefore we have
$$C\cap \overline{C\cap G_k\cdot x}=C\cap \overline{C\cap G_k\cdot x}^\tau.$$
In fact, we have
$$C\cap \overline{G_k\cdot x}=C\cap \overline{G_k\cdot x}^\tau$$
for every $x\in C$ because $G(x,C)$ is comeager, thus dense in $G_k$ (see {\bf (I)} in the proof of \cref{lm:weak boundedness}).

Consider the canonical identification between $2^{\I N\times \I N}$ and $2^{k\times \I N}\times 2^{\I N \times \I N}$.
Another use of \cite[Theorem~8.41]{kechrisclassical} gives $y\in 2^{k\times \I N}$ such that
$$Y_k=\{y^+\in 2^{\I N\times \I N}:(y,y^+)\in O\cap C\cap D_k\}$$
is comeager in $2^{\I N\times \I N}$ with respect to the product topology.
Pick $y^+_0\in Y_k$ and note that the set
$$\{y^+\in 2^{\I N\times \I N}:((y,y^+_0),(y,y^+))\in \I E_0^{\I N}\}$$
is meager.
Therefore there is $y^+_1\in Y_k$ such that $((y,y^+_0),(y,y^+_1))\not\in \I E_0^{\I N}$.
However, we have
$$(y,y^+_1)\in C\cap \overline{G_k\cdot (y,y^+_0)}=C\cap \overline{G_k\cdot (y,y^+_0)}^{\tau}$$
and that is a contradiction.
\end{proof}

\begin{corollary}\label{cor:bounded implies no E_3}
Let $X$ be a Polish $G$-space such that $E^X_G$ is Borel.
Suppose that $X$ admits a bounded $G$-Polish topology $\tau$.
Then $\I E_0^{\I N}\not \le_B E^X_G$. 
\end{corollary}
\begin{proof}
If $\I E_0^{\I N}\le_B E^X_G$, then $2^{\I N\times \I N}$ admits a finer bounded $(2^{<\I N})^{\I N}$-Polish topology by \cref{th:bounded downward closed}.
That contradicts \cref{th:E_3 not bounded}.
\end{proof}

Hjorth and Kechris have shown that if $G$ is a non-archimedean\footnote{A topological group is non-archimedean if it admits a base at the identity made of clopen subgroups.} tsi Polish group, $X$ is a Polish $G$-space such that $E^X_G$ is a Borel equivalence relation, then either $E^X_G$ is essentially countable or $\I E_0^{\I N}\le_B E^X_G$ , see~\cite[Theorem~8.1]{HjorthKechris}.
This gives immediately.

\begin{corollary}\label{cor:HjorthKechris}
Let $G$ be a tsi non-archimedean Polish group and $X$ be a Polish $G$-space such that $E^X_G$ is Borel.
Then the following are equivalent:
\begin{itemize}
	\item $E^X_G$ is essentially countable,
	\item there is a finer bounded $G$-Polish topology on $X$,
	\item $\I E_0^{\I N}\not\le_B E^X_G$.
\end{itemize}
\end{corollary}

\section{Property (IC)}\label{sec:PropIC}

We define a combinatorial property that characterizes classification by countable structures for equivalence relations induced by actions of tsi Polish groups, this is proved later in \cref{sec:ProofOfCCS}.
Similar property was considered, e.g., in \cite[Chapter~15.2]{Kanovei} under the name ``Grainy sets''.
This property, as well as classification by countable structures, does not depend on the underlying Polish topology on the space $X$, i.e., it depends only on the Borel $\sigma$-algebra.
We note that every action of \emph{non-archimedean} Polish group trivially satisfies this property.

Recall that the dihypergraphs $\fH_{k,m}$ are defined in \cref{def:dihyp}.

\begin{definition}[Property (IC)]
Let $X$ be a Polish $G$-space and $B\subseteq X$ be a $G$-invariant Borel set.
We say that $B$ satisfies \emph{Property (IC)} if there is a sequence of Borel sets $(A_{k,l})_{k,l\in\I N}$ such that
\begin{itemize}
    \item for every $k,l\in \I N$ there is $m(k,l)\in \I N$ such that $A_{k,l}$ is $\fH_{k,m(k,l)}$-independent,
    \item $B=\bigcup_{l\in \I N} A_{k,l}$ for every $k\in \I N$.
\end{itemize}
We say that the Polish $G$-space $X$ or the equivalence relation $E^X_G$ satisfy Property (IC) if $X$ satisfies property (IC).
\end{definition}

We start by showing that Property (IC) is orthogonal to turbulence for actions of any Polish group.

\begin{theorem}\label{th:turbulence vs IC}
Let $X$ be a Polish $G$-space that satisfies Property (IC).
Then the action is not turbulent.
\end{theorem}
\begin{proof}
Suppose that the action is turbulent.
Let $D\subseteq X$ be a Borel comeager set such that $A_{k,l}\cap D$ is relatively open in $D$ for every $k,l\in \I N$.
This can be done using \cite[Proposition~8.26]{kechrisclassical}.
It follows from \cite[Theorem~16.1]{kechrisclassical} and \cite[Theorem~8.41]{kechrisclassical} that
$$D'=\{x\in D:\forall^* g\in G \ g\cdot x\in D\}$$
is a Borel comeager subset of $X$.

Pick $x\in D'$.
Note that $G(x,D')$ is comeager in $G$.
We show that $G\cdot x=[x]_{E^X_G}$ is nonmeager.
Suppose that $G\cdot x$ is meager.
Then there are closed nowhere dense sets $(F_r)_{r\in \I N}$ such that $G\cdot x\subseteq \bigcup_{r\in \I N}F_r$.
Note that $G(x,F_r)$ is closed for every $r\in \I N$ and $G=\bigcup_{r\in \I N} G(x,F_r)$.
By \cite[Proposition~8.26]{kechrisclassical} there is an index $r\in \I N$ such that $G(x,F_r)$ contains an open set.
This implies that there is $g\in G$ and $k\in\I N$ such that $\Delta_k\cdot g\subseteq G(x,F_r)$ and $y=g\cdot x\in D'$.
Let $l\in \I N$ such that $y\in A_{k,l}$.
Note that
$$\overline{\Delta_k\cdot y}=\overline{\Delta_k\cdot g\cdot x}\subseteq F_r$$
because $F_r$ is closed.

Use the definition of $D$ to find an open set $U$ such that $U\cap D'=A_{k,l}\cap D'$.
Consider the local orbit $\mathcal{O}(y,U,\Delta_{m(k,l)})$ and pick $z\in \mathcal{O}(y,U,\Delta_{m(k,l)})$.
By the definition, there is $w\in U^{<\I N}$ such that $w_0=y$, $w_{|w|-1}=z$ and $(w_i,w_{i+1})\in R^X_{\Delta_{m(k,l)}}$ for every $i<|z|-1$.
Let $P\subseteq X$ be an open neighborhood of $z$.
Note that $G(y,U)$, $G(y,P)$ are open and $G(y,D')$ is comeager, in particular, dense in $G(y,U)$.
Therefore we can find a sequence $w'\in U^{<\I N}$ such that $|w|=|w'|$, $w'_0=y$, $w'_i\in U\cap D'$ for every $i<|w'|$, $(w'_i,w'_{i+1})\in R^X_{\Delta_{m(k,l)}}$ for every $i<|w'|-1$ and $w'_{|w'|-1}\in P$.
Note that we have
$$w'_i\in U\cap D'=A_{k,l}\cap D'\subseteq A_{k,l}$$
for every $i<|w'|$.
The set $A_{k,l}$ is $\mathcal{H}_{k,m(k,l)}$-independent and therefore $(y,w'_{|w'|-1})\in R^X_{\Delta_k}$.
This implies that $\Delta_k\cdot y\cap P\not=\emptyset$ and consequently that
$$\mathcal{O}(y,U,\Delta_{m(k,l)})\subseteq \overline{\Delta_k\cdot y}.$$
Therefore $F_r$ contains an open set by the assumption that the action is turbulent, i.e., $\mathcal{O}(y,U,\Delta_{m(k,l)})$ is somewhere dense.
This shows that $[x]_{E^X_G}$ is nonmeager and that contradicts the definition of turbulence.
\end{proof}

We conclude this section by stating that Property (IC) is a stronger condition than classification by countable structures for tsi Polish groups.

\begin{theorem}\label{th:IC ->CCS}
Let $G$ be a tsi Polish group and $X$ be a Polish $G$-space such that $E^X_G$ is Borel.
Suppose that $E^X_G$ satisfies Property (IC). 
Then $E^X_G$ is classifiable by countable structures.
\end{theorem}
\begin{proof}
An elementary proof of this statement follows from \cite[Definition~3.3.6,~Proposition~3.3.7,~Theorem~3.3.8]{GrebikPhD}.

Alternative approach that does not need the assumption that $E^X_G$ is Borel is to appeal to \cite[Theorem~13.18]{KechrisTurbulence} and \cref{th:turbulence vs IC}.
\end{proof}

\section{Proof of \cref{thm:PropICandGbounded}}\label{sec:ProofOfCCS}

We show that Property (IC) together with the existence of bounded $G$-Polish topology implies that the equivalence relation is essentially countable.
Note that it follows from \cref{cor:EC implies bounded} that if a Borel equivalence relation induced by an action of Polish group is essentially countable, then it admits a bounded $G$-Polish topology.
The following is a formal formulation of \cref{thm:PropICandGbounded} for tsi Polish groups.

\begin{theorem}\label{th:IC+bounded=EC}
Let $G$ be a tsi Polish group $X$ be a Polish $G$-space such that $E^X_G$ is Borel.
Then the following are equivalent:
\begin{itemize}
    \item [(A)] $E^X_G$ satisfies Property (IC) and there is a finer bounded $G$-Polish topology on $X$,
    \item [(B)] $E^X_G$ is essentially countable.
\end{itemize}
Moreover, (A) implies (B) for any Polish group $G$.
\end{theorem}

The strategy for showing (A) implies (B) is as follows.
By the assumptions we fix a bounded $G$-Polish topology $\tau$ and a sequence $(A_{k,l})_{k,l\in \I N}$ of Borel $\fH_{k,(m(k,l))}$-independent sets.
We define an equivalence relation $F_{k,l}$ on $A_{k,l}$ as
$$(x,y)\in F_{k,l} \ \Leftrightarrow \ \exists z\in \fJ(\Delta_{m(k,l)})\cap (A_{k,l})^{<\I N} \ z_0=x, \ z_{|z|-1}=y$$
for every $x,y\in A_{k,l}$ and $k,l\in \I N$.
Note that if $x\in A_{k,l}$, then $[x]_{F_{k,l}}\subseteq \Delta_k\cdot x$.

Since $\tau$ is $G$-bounded we find for each $x\in X$ natural numbers $k,l\in \I N$ such that $\overline{[x]_{F_{k,l}}}^\tau\subseteq [x]_{E^X_G}$.
By \cite[Theorem~12.13]{kechrisclassical} there is a Borel selector $S$ that picks from every nonempty $\tau$-closed set one of its elements.
We define
$$x\mapsto \overline{[x]_{F_{k,l}}}^\tau\mapsto S\left(\overline{[x]_{F_{k,l}}}^\tau\right)\in [x]_{E^X_G}.$$
The idea is to show that this is a Borel map with range that is a countable complete section.
This can be done once we pass from $X$ to a suitable Borel $G$-lg comeager set $C$.
Formal proof follows.

\begin{proof}[Proof of \cref{th:IC+bounded=EC} and \cref{thm:PropICandGbounded}]
We start with (B) $\Rightarrow$ (A).
We mentioned above that by \cref{cor:EC implies bounded} we have that (B) implies the existence of a bounded $G$-Polish topology for every Polish group $G$.
To show that (B) implies Property (IC), we need to assume that $G$ is tsi.
In that case, we either use \cref{thm:CCS=ICextended}, or exploit the fact that essential countability is equivalent to $\sigma$-lacunarity \cite{GrebikLacunary} together with \cref{pr:Borel independent}.

Next, we formalize the ideas for (A) $\Rightarrow$ (B) that are sketched before the proof.
We fix a bounded $G$-Polish topology $\tau$ and a sequence $(A_{k,l})_{l\in \I N}$ of $\fH_{k,m(k,l)}$-independent Borel sets.
We may assume that $(A_{k,l})_{l\in \I N}$ are pairwise disjoint for every $k\in\I N$.
We denote as $x(k)$ the unique $l\in \I N$ such that $x\in A_{k,l}$, i.e., $x\in A_{k,x(k)}$ for every $k\in \I N$.
Let $S$ be a Borel selector that assigns to a non-empty $\tau$-closed subset $K\subseteq X$ its element, i.e., a Borel map $K\mapsto S(K)\in K$ from the standard Borel space of $\tau$-closed non-empty subsets to $X$, see~\cite[Theorem~12.13]{kechrisclassical}.
The proof consists of five steps.

{\bf (I)}.
Fix $k\in\I N$ and use \cref{lm:rel open} for the sequence $(A_{k,l})_{l\in \I N}$ to get a Borel $G$-lg comeager set $C_k$.
Define $C=\bigcap_{k\in\I N} C_k$.
Then we have that $C$ is a Borel $G$-lg comeager set and for every $k\in \I N$ and $x\in C\cap A_{k,x(k)}$ there is an open neighborhood $\Delta$ of $1_G$ such that 
$$C\cap \Delta\cdot x\subseteq C_k\cap \Delta\cdot x\subseteq A_{k,x(k)}$$
by \cref{lm:rel open}.

{\bf (II)}.
Let $k,l\in \I N$.
Define the equivalence relation $F_{k,l}$ on $C\cap A_{k,l}$ as
$$(x,y)\in F_{k,l} \ \Leftrightarrow \ \exists z\in \fJ(\Delta_{m(k,l)}) \cap (C\cap A_{k,l})^{<\I N} \ z_0=x, \ z_{|z|-1}=y$$
for every $x,y\in C\cap A_{k,l}$.
We show that $F_{k,l}$ is Borel and that every $E^{X}_G$-class contains at most countably many $F_{k,l}$-classes.

Let $x\in X$ and $y\in [x]_{E^X_G}\cap C\cap A_{k,l}$.
Then $G(x,[y]_{F_{k,l}})$ is relatively open in $G(x,C)$ by the properties of $C$ from {\bf (I)}.
This shows that each $E^X_G$-class contains at most countably many $F_{k,l}$-classes because $G(x,C)$ is a separable space.

Set $R^0$ for the restriction of the relation $R^X_{\Delta_{k,l}}$ to $C\cap A_{k,l}$.
It follows from \cite[Theorem~7.1.2]{BeckerKechris} and the assumption that $E^X_G$ is Borel that $R_{k,l}$ is Borel.
Inductively on $i\in \I N$ define relations $R^i$ on $C\cap A_{k,l}$ as
$$(x,y) \in R^{i+1}\ \Leftrightarrow \ \exists^* g\in G \ (x,g\cdot x)\in R^i \ \wedge \ (g\cdot x,y)\in R^X_{\Delta_{m(k,l)}}.$$
It follows inductively from \cite[Theorem~16.1]{kechrisclassical} that $R=\bigcup_{i\in \I N} R^i$ is a Borel subset of $C\cap A_{k,l}\times C\cap A_{k,l}$.

We show that $F_{k,l}=R$.
It is easy to see that $R\subseteq F_{k,l}$.

\begin{claim}\label{cl:arifical}
Suppose that $(x,y)\in R^i$.
Then there is an open neighborhood $\Delta$ of $1_G$ (that depends on $x,y$ and $i$) such that $(x,h\cdot y)\in R^i$ for every $h\in \Delta$ such that $h\cdot y\in C$.
\end{claim}
\begin{proof}
Let $i=0$ and pick $(x,y)\in R^0$.
There is $g\in \Delta_{m(k,l)}$ such that $y=g\cdot x$.
Use {\bf (I)} to find $\Delta$ such that $\Delta\cdot g\subseteq \Delta_{m(k,l)}$ and $C\cap \Delta\cdot y\subseteq A_{k,l}$.
Then for every $h\in \Delta$ such that $h\cdot y\in C$ we have $h\cdot y\in A_{k,l}$ and since $h\cdot g\in \Delta_{m(k,l)}$ we conclude that $(x,h\cdot y)\in R^0$.

Let $(x,y)\in R^{i+1}$ and $z\in A_{k,l}$ be such that $(x,z)\in R^i$ and $y\in \Delta_{m(k,l)}\cdot z$.
By the inductive hypothesis there is $\Delta'$ an open neighborhood of $1_G$ such that $(x,h\cdot z)\in R^i$ whenever $h\cdot z\in C$.
Write $y=g_0\cdot z$ where $g_0 \in \Delta_{m(k,l)}$.
Using {\bf (I)} we find an open neighborhood $\Delta$ of $1_G$ such that $\Delta\cdot g_0\subseteq \Delta_{m(k,l)}$ and $C\cap \Delta\cdot y\subseteq A_{k,l}$.
We claim that $\Delta$ works as required.
Let $h\in \Delta$ such that $h\cdot y\in C$.
Note that
$$P_h=\{a\in \Delta':h\cdot g_0\cdot a^{-1}\in \Delta_{m(k,l)}\}$$
is an open neighborhood of $1_G$.
Let $a\in P_h$ be such that $a\cdot z\in C$, note that there are nonmeager many such $a\in P_h$ because $G(z,C)$ is comeager in $P$.
Then we have $(x,a\cdot z)\in R^i$ and
$$h\cdot y=h\cdot g_0\cdot z=h\cdot g_0\cdot a^{-1}\cdot (a\cdot z)\in \Delta_{m(k,l)}\cdot (a\cdot z).$$
This shows that $(x,h\cdot y)\in R^{i+1}$ and the claim follows.
\end{proof}

Suppose that $F_{k,l}\setminus R\not =\emptyset$ and pick $(x,y)\in F_{k,l}\setminus R$ such that the witness $z\in (C\cap A_{k,l})^{<\I N}$ from the definition of $F_{k,l}$ has minimal length.
It follows that $|z|>2$ and $(x,z_{|z|-2})\in R^i$ for some $i\in \I N$.
Use \cref{cl:arifical} to find an open neighborhood $\Delta$ of $1_G$ that satisfies $(x,h\cdot z_{|z|-2})\in R^i$ for every $h\in \Delta$ such that $h\cdot z_{|z|-2}\in C$.
Let $a\in \Delta_{m(k,l)}$ be such that $y=a\cdot z_{|z|-2}$.
The set 
$$P=\{h\in \Delta:a\cdot h^{-1}\in \Delta_{m(k,l)}\}$$
is an open neighborhood of $1_G$.
Then we have
$$y=a\cdot z_{|z|-2}=a\cdot h^{-1} \cdot (h\cdot z_{|z|-2})\in \Delta_{m(k,l)}\cdot (h\cdot z_{|z|-2})$$
for every $h\in P$.
Note that $G(z_{|z|-2},C)$ is comeager in $P$ and that shows that $(x,y)\in R^{i+1}$, a contradiction.

{\bf (III)}.
Fix $k\in \I N$ and define
$$K_{k}(x)=\overline{[x]_{F_{k,x(k)}}}^\tau$$
for every $x\in C$.
Let $U\subseteq X$ be a $\tau$-open set.
Then we have
$$U\cap K_k(x)\not=\emptyset \ \Leftrightarrow \ \exists g\in G \ g\cdot x\in [x]_{F_{k,x(k)}}\cap U \ \Leftrightarrow \ \exists^*g\in G \ g\cdot x\in [x]_{F_{k,x(k)}}\cap U$$
where the last equivalence follows from properties of $C$ from {\bf (I)}.
Combination of {\bf (II)} and \cite[Theorem~16.1]{kechrisclassical} implies that $K_k$ is a Borel map, see~\cite[Section~12]{kechrisclassical}.

{\bf (IV)}.
Put $S_k=S\circ K_k:X\to X$.
Then $S_k$ is a Borel map for every $k\in \I N$ by {\bf (III)}.
Define 
$$D_k=\{(x,S_k(x)):x\in C\}\cap E^{X}_G\subseteq X\times X.$$
It is easy to see that $D_k$ is a Borel set.
We show that $p_2(D_k)$, the projection to the second coordinate, is a Borel countable section of $E^{X}_G$.

First observe that for $x\in C$ we have $(x,S_k(x))\in D_k$ if and only if $(y,S_k(y))\in D_k$ for every $y\in [x]_{F_{k,x(k)}}$ and $S(x)=S(y)\in [x]_{E^X_G}$.
It follows from {\bf (II)} that there are at most countable many $F_{k,l}$-classes within each $E^X_G$-class and therefore $p_2(D_k)$ is a countable section of $E^X_G$.
To see that $p_2(D_k)$ is Borel note that by the properties of $C$ from {\bf (I)} and the definition of $S_k$ and $D_k$ we have
$$z\in p_2(D_k) \ \Leftrightarrow \ \exists g\in G \ g\cdot z \in p_1(D_k) \ \Leftrightarrow \ \exists^* g\in G \ g\cdot z \in p_1(D_k)$$
where $p_1(D_k)$ is the projection to the first coordinate.
Since $p_1(D_k)$ is Borel we have that $p_2(D_k)$ is Borel by \cite[Theorem~16.1]{kechrisclassical}.

{\bf (V)}.
Finally, we need to show that $\bigcup_{k\in\I N} p_2(D_k)$ is a complete section.
Let $x\in C$.
It follows from the definition of Property (IC) and $\tau$ that there is $k\in \I N$ such that 
$$\overline{[x]_{F_{k,x(k)}}}^\tau\subseteq \overline{\Delta_k\cdot x}^\tau\subseteq [x]_{E^X_G}.$$
Then we have $S_k(x)\in [x]_{E^X_G}$ and consequently $(x,S_k(x))\in D_k$.
This shows that $p_2(D_k)\cap [x]_{E^X_G}\not=\emptyset$.
\end{proof}

\section{Proof of \cref{thm:CCS=IC}}

In this section we show that classification by countable structures is equivalent to Property (IC) for Borel equivalence relations induced by actions of tsi Polish groups.
In the proof we use the dihypergraph variant of $\I G_0$-dichotomy \cref{thm:HypG0}.
Namely, we show that (1.) implies Property (IC), and that the continuous map from (2.) can be refined to a reduction from an equivalence relation from $\mathfrak{B}$, see \cref{sec:BASE}.
The technical results that we need are collected in \cref{app:technical} and \cref{app:Refine}.
The following statement is a formal reformulation of \cref{thm:CCS=IC}.

\begin{theorem}\label{thm:CCS=ICextended}
Let $G$ be a tsi Polish group and $X$ be a Polish $G$-space such that $E^X_G$ is Borel.
Then the following are equivalent:
\begin{enumerate}
    \item $E^X_G$ satisfies Property (IC),
    \item $E^X_G$ is classifiable by countable strutures.
\end{enumerate}
Moreover, the conditions are not satisfied if and only if there is $E\in \mathfrak{B}$ such that $E\le_B E^X_G$.
\end{theorem}
\begin{proof}[Proof of \cref{thm:CCS=IC} and of \cref{thm:CCS=ICextended}]
It follows from \cref{th:IC ->CCS} that (1) implies (2).
Moreover, by \cref{th:base for nonCCS} we have that the latter condition implies the former in the additional part of the statement.
Altogether, it is enough to show that if $E^X_G$ does not satisfy Property (IC), then there is $E\in \mathfrak{B}$ such that $E\le_B E^X_G$.

Fix $k\in \I N$ and consider the dihypergraphs $(\fH_{k,m})_{m\in \I N}$.
Note that $\fH_{k,m+1}\subseteq \fH_{k,m}$ for every $m\in \I N$.
Then exactly one alternative in \cref{thm:HypG0} holds.
Suppose that the first one is satisfied for every $k\in \I N$.
Then it follows directly from the definition that $E^X_G$ satisfies Property (IC).
Since we assume that this is not the case, there must be $k\in \I N$, a finitely uniformly branching tree $T'$, a dense sequence $(s'_m)_{m\in \I N}\subseteq T'$ such that $|s'_m|=m$ and a continuous homomorphism $\varphi:[T'] \to X$ from $(\I G^{T'}_{s'_m})_{m\in \I N}$ to $(\fH_{k,m})_{m\in \I N}$.

Next, we refine $\varphi$ to find $E\in \mathfrak{B}$.
The following is the main technical result, the proof can be found in \cref{app:Refine}, see also \cref{sec:pseuodmetric} for corresponding definitions.
We note that it uses crucially that $G$ is tsi.

\begin{lemma}\label{lm:refinement}
There is a finitely uniformly branching tree $T$, a dense sequence $(s_m)_{m\in \I N}\subseteq T$ such that $|s_m|=m$ for every $m\in \I N$ and a continuous homomorphism $\phi:[T]\to X$ from $(\I G^T_{s_m})_{m\in \I N}$ to $(\fH_{k,m})_{m\in \I N}$ such that ${\bf d}_{\phi}$, defined as in \cref{pr:Borel pseudo}, is a uniform Borel pseudometric.
Moreover, $\phi=\varphi\circ \zeta$ where $\zeta:[T]\to [T']$ is a continuous map.
\end{lemma}

The rest of the proof consists of four steps.

{\bf (I)}.
Let $m\in \I N$, $s,t\in T$ such that $m=|s|=|t|$, $i,j<l^T_m$ and $x,y\in [T_{s^\frown (i)}]=[T_{s^\frown (j)}]$.
Then
\begin{equation}\tag{*}\label{eq:smthn1}
|{\bf d}_\phi(s^\frown (i)^\frown x,s^\frown (j)^\frown x)-{\bf d}_\phi(t^\frown (i)^\frown y,t^\frown (j)^\frown y|<\frac{1}{2^{m-1}}.
\end{equation}
We use that ${\bf d}_\phi$ is uniform.
Namely, we have
$$|{\bf d}_\phi(s^\frown ((i)^\frown y),s^\frown ((j)^\frown y))-{\bf d}_\phi(t^\frown ((i)^\frown y),t^\frown ((j)^\frown y)|<\frac{1}{2^{m}}$$
$$|{\bf d}_\phi((s^\frown (i))^\frown x,(s^\frown (j))^\frown x)-{\bf d}_\phi((s^\frown (i))^\frown y,(s^\frown (j))^\frown y|<\frac{1}{2^{m+1}}$$
and that gives the estimate by the triangle inequality.

In particular, we have
$$\I E^T_0\subseteq F_{{\bf d}_\phi}=(\phi^{-1}\times \phi^{-1})(E^X_G)$$
because $\phi$ is a homomorphism from $(\I G^T_{s_m})_{m\in \I N}$ to $(\fH_{k,m})_{m\in \I N}$ and $|s_m|=m$ for every $m\in \I N$.

{\bf (II)}.
The Borel equivalence relation $F_{{\bf d}_\phi}$ is meager in $[T]\times [T]$.
Otherwise there is $\alpha\in [T]$ such that $[\alpha]_{\bf d}$ is comeager in $[T]$ by $\mathbb{E}^T_0\subseteq F_{{\bf d}_\phi}$ and \cite[Theorem~8.41]{kechrisclassical}.
It follows from (3) in the definition of Borel pseudometric that there are Borel sets $(U_l)_{l\in \I N}$ such that $\bigcup_{l\in \I N} U_l=[\alpha]_{F_{\bf d}}$ and 
$${\bf d}_{\phi}(\alpha,\beta)<\frac{1}{2^{k}}$$
for every $l\in \I N$ and $\alpha,\beta\in U_l$.
Using \cite[Proposition~8.41]{kechrisclassical} and the density of $(s_m)_{m\in\I N}$ we find $m,l\in \I N$ such that $U_l$ is comeager in ${s_m}^\frown [T_{s_m}]$.
This gives $x\in [T_{{s_m}^\frown (0)}]=[T_{{s_m}^\frown (l^T_m-1)}]$ such that
$${s_m}^\frown (0)^\frown x,{s_m}^\frown (l^T_m-1)^\frown x\in U_l.$$
Since $\phi$ is a homomorphism from $\mathbb{G}^T_{s_m}$ to $\mathcal{H}_{k,m}$, we have $(\phi({s_m}^\frown (i)^\frown x))_{i<l^T_m}\in \mathcal{H}_{k,m}$.
Consequently,
$$(\phi({s_m}^\frown (0)^\frown x),\phi({s_m}^\frown (l^T_m-1)^\frown x)\not \in R^X_{\Delta_k},$$
i.e.,
$${\bf d}_\phi({s_m}^\frown (0)^\frown x,{s_m}^\frown (l^T_m-1)^\frown x)>\frac{1}{2^k}.$$
This contradicts the choice of $x\in[T_{{s_m}^\frown (0)}]$ and we conclude that $F_{{\bf d}_\phi}$ is a meager equivalence relation.

{\bf (III)}.
Write ${\bf 0}\in [T]$ for the sequence $(0,0,0,\dots)$.
Let $m\in \I N$.
Since $(({s_m}^\frown (i)^\frown {\bf 0})_{i<l^T_m},{\bf d}_\phi)$ is a finite pseudometric space we find a metric space $(Z_m,\mathfrak{d}_m)$ where $Z_m=\{0,1,\dots, l^T_m-1\}$ and 
$$|{\bf d}_\phi({s_m}^\frown (i)^\frown {\bf 0},{s_m}^\frown (j)^\frown {\bf 0})- \mathfrak{d}_m(i,j)|<\frac{1}{2^{m-1}}$$
for every $i,j<l^T_m$.
We have
$$\frac{1}{2^k}-\frac{1}{2^{m-1}}\le \mathfrak{d}_m(0,l^T_m-1)\le r(Z_m,\mathfrak{d}_m)$$
and $j(Z_m,\mathfrak{d}_m)<\frac{1}{2^{m-2}}$ because $\phi$ is a homomorphism from $\I G_{s_m}$ to $\mathcal{H}_{k,m}$.

This implies immediately that $\mathcal{Z}=((Z_m,\mathfrak{d}_m))_{m\in \I N}$ is a non-trivial sequence of finite metric spaces.
Consider the bijective homeomorphism
$$\eta:\prod_{m\in \I N} Z_m\to [T]$$
that is defined as
$$\eta(x)(m)=i \ \Leftrightarrow \ x(m)=i.$$
If $E_\mathcal{Z}\subseteq E=(\eta^{-1}\times \eta^{-1})(F_{{\bf d}_\phi})$, then we are done because  $E\in \mathfrak{B}$ by {(II)} and $\phi\circ \eta$ is a reduction from $E$ to $E^X_G$.

{\bf (IV)}.
Suppose that $E_\mathcal{Z}\not\subseteq E=(\eta^{-1}\times \eta^{-1})(F_{{\bf d}_\phi})$ in {(III)}.
By the definition, we find $x,y\in \prod_{m\in \mathbb{N}} Z_m$ such that
$$\mathfrak{d}_m(x(m),y(m))\to 0$$
and $(\eta(x),\eta(y))\not \in F_{{\bf d}_\phi}$.
Set $\alpha=\eta(x)$ and $\beta=\eta(y)$.

We have $|\{m\in \I N:\alpha(m)\not =\beta(m)\}|=\aleph_0$ because $\I E^T_0\subseteq F_{{\bf d}_{\phi}}$.
Let $(m_l)_{l\in \I N}$ be an increasing enumeration of $\{m\in \I N:\alpha(m)\not = \beta(m)\}$ and set ${\bf 0}_l=\alpha(m_l)$, ${\bf 1}_l=\beta(m_l)$ for every $l\in \I N$.
There is a sequence $(t_l)_{l\in \I N}\subseteq \I N^{<\I N}$ such that 
$$\alpha={t_0}^\frown {{\bf 0}_0}^\frown {t_1}^\frown {{\bf 0}_1}^\frown \dots \ \& \ \beta={t_0}^\frown {{\bf 1}_0}^\frown {t_1}^\frown {{\bf 1}_1}^\frown \dots$$
Define $\Gamma:2^{<\I N} \to T$ as
$$\Gamma(s)={t_0}^\frown {{\bf s(0)}}^\frown {t_1}^\frown {{\bf s(1)}}^\frown\dots ^\frown {{\bf s(|s|-1)}}^\frown t_{|s|}\in T$$
for every $s\in 2^{<\I N}$, where ${\bf s(j)}={\bf 0}_j$ if $s(j)=0$ and ${\bf s(j)}={\bf 1}_j$ if $s(j)=1$.
Write $\widetilde{\Gamma}:2^\I N\to [T]$ for the extension of $\Gamma$ to $2^{\I N}$.
It is easy to see that $\widetilde{\Gamma}$ is a well defined continuous map.

Set ${\bf d}={\bf d}_{\phi\circ \widetilde{\Gamma}}$ and $E:=F_{{\bf d}_{\phi\circ \widetilde{\Gamma}}}$.
It is easy to see that ${\bf d}$ is a uniform Borel pseudometric, this follows from the definition of $\Gamma$.
Moreover, $(\widetilde{\Gamma}^{-1}(\alpha),\widetilde{\Gamma}^{-1}(\beta))\not\in F_{\bf d}$.
Consequently by \cref{th:uniform nonmeager is everything}, we have that $E$ is meager in $2^{\I N}\times 2^{\I N}$.

To finish the proof we define a tall lsc submeasure $\Theta$ such that $E_{\Theta}\subseteq E$.
Indeed, then we have $E\in \mathfrak{B}$ and, clearly, $\phi\circ\widetilde{\Gamma}$ is a reduction from $E$ to $E^X_G$.

Recall that $\fP(\I N)$ is the power set of $\I N$.
Let $M\in \mathcal{P}(\I N)$ be finite.
Define
\begin{equation*}
\begin{split}
    \Theta(M)= & \ \sup\left\{{\bf d}(a,b):a,b\in 2^{\I N} \ \{\ell\in \I N:a(\ell)\not=b(\ell)\}\subseteq M \ \right\} \\
    = & \ \sup\left\{{\bf d}_\phi(\widetilde{\Gamma}(a),\widetilde{\Gamma}(b)):a,b\in 2^{\I N} \ \{m\in \mathbb{N}:\widetilde{\Gamma}(a)(m)\not=\widetilde{\Gamma}(b)(m)\}\subseteq (m_\ell)_{\ell\in M} \ \right\}.
\end{split}
\end{equation*}
Let $M\in \mathcal{P}(\mathbb{N})$ be infinite.
Then we define $\Theta(M)=\lim_{l\to\infty} \Theta(M\cap l)$.

{\bf (a)}.
We show that $\Theta$ is a tall lsc submeasure.
\begin{itemize}
    \item It is easy to see that $\Theta$ is monotone, $\Theta(\emptyset)=0$ and $\Theta(M)=\lim_{\ell\to \infty}\Theta(M\cap \ell)$ for every $M\in \mathcal{P}(\I N)$.
    \item Let $M,N\in \fP(\I N)$ be finite and $a,b\in 2^{\I N}$ such that $\{\ell\in \I N:a(\ell)\not b(\ell)\}\subseteq M\cup N$.
    Set $c\in 2^{\I N}$ to be equal to $b$ on $M$ and to to $a$ on $N\setminus M$.
    Then we have
    \begin{equation*}
        \begin{split}
            {\bf d}(a,b)\le {\bf d}(a,c)+{\bf d}(c,b)\le \Theta(M)+\Theta(N),
        \end{split}
    \end{equation*}
    becuase ${\bf d}$ is a pseudometric.
    Since this holds for every such $a,b\in 2^{\I N}$, we conclude that $\Theta(M\cup N)\le \Theta(M)+\Theta(N)$.
    For infinite $M,N$ the conclusion holds by taking the limit in the definition.
    \item Let $\ell\in \I N$ and $a,b\in 2^{\I N}$ be such that $a(\ell')\not=b(\ell')$ only when $\ell=\ell'$.
    Observe that this implies that $\widetilde{\Gamma}(a)$ and $\widetilde{\Gamma}(b)$ differ only at $m_\ell$.
    We have
    \begin{equation*}
        \begin{split}
            {\bf d}(a,b)= & \ {\bf d}_\phi(\widetilde{\Gamma}(a),\widetilde{\Gamma}(b)) \\
            \le & \ {\bf d}_\phi({s_{m_\ell}}^\frown{\alpha(m_\ell)}^\frown {\bf 0},{s_{m_\ell}}^\frown{\beta(m_\ell)}^\frown {\bf 0})+\frac{1}{2^{m_\ell-1}} \\
            \le & \ \mathfrak{d}_{m_\ell}(\alpha(m_\ell),\beta(m_\ell))+\frac{1}{2^{m_\ell-2}}
        \end{split}
    \end{equation*}
    by \eqref{eq:smthn1} and the definition of $\mathfrak{d}_{m_\ell}$ in (III).
    This shows that $\Theta(\{\ell\})<+\infty$ and $\lim_{\ell\to \infty}\Theta(\{\ell\})=0$ by the choice of $\alpha$ and $\beta$.
    This shows that $\Theta$ is a tall lsc submeasure.
\end{itemize}

{\bf (b)}
Let $a,b\in 2^{\I N}$ be such that $(a,b)\in E_{\Theta}$.
We show that $(a,b)\in E$.
Set $X=\{j\in \I N:a(j)\not=b(j)\}$.
Then we have that $\lim_{\ell\to \infty} \Theta(X\setminus \ell)=0$ by the definition of $E_\Theta$.
For every $\ell\in \I N$, define $a_\ell(j)=b(j)$ for every $j<\ell$ and $a_\ell(j)=a(j)$ for every $j\ge \ell$.
We have $(a_\ell,a)\in \I E_0$ for every $\ell\in \I N$ and $a_\ell\to b$.

We have that $\widetilde{\Gamma}(a_\ell)\to \widetilde{\Gamma}(b)$ by continuity of $\widetilde{\Gamma}$ and it is easy to see that
$$\left(\widetilde{\Gamma}(a_\ell),\widetilde{\Gamma}(a)\right)\in \I E^T_0\subseteq F_{{\bf d}_\phi}.$$
Let $\ell\le r\le s\in \I N$.
We have that $\{j\in \I N:a_r(j)\not =a_s(j)\}= X\cap \{r,\dots,s-1\}\subseteq X\setminus \ell$.
Consequently, by the definition of $\Theta$ and our assumption, we have
\begin{equation*}
{\bf d}_{\phi}\left(\widetilde{\Gamma}(a_r),\widetilde{\Gamma}(a_s) \right)\le \Theta(X\cap \{r,\dots,s-1\})\le \Theta(X\setminus \ell)\to 0.
\end{equation*}
In particular, $(\widetilde{\Gamma}(a_\ell))_{\ell\in \I N}$ is a ${\bf d}_\phi$-Cauchy sequence.
By the definition of Borel pseudometric we have ${\bf d}_\phi(\widetilde{\Gamma}(a_\ell),\widetilde{\Gamma}(b))\to 0$.
In particular, $(\widetilde{\Gamma}(a),\widetilde{\Gamma}(b))\in F_{\bf d_\phi}$ and, consequently, $(a,b)\in E$.
\end{proof}

\section{Proof of \cref{thm:E_3dichotomy}}

In this section we combine the characterization of classification by countable structures from previous section together with a result of Miller \cite{MillerLacunary} to show that if a Borel equivalence relation $E^X_G$ induced by an action of tsi Polish group $G$ admits classification by countable structures, then it is essentially countable if and only if $\I E_0^{\I N}\not\le_B E^X_G$.
In another words, under the assumption of classification by countable structures we have that $\I E_0^{\I N}$ is the canonical obstruction for EC.
We note that in the previous results \cite{HjorthKechris,MillerLacunary} the corresponding statement is proved for tsi \emph{non-archimedean} Polish groups.

The strategy of the proof combines two variants of the $\I G_0$-dichotomy.
First one is hidden in the implication that classification by countable structures implies Property (IC) and the second one is the characterization of $\sigma$-lacunarity given by Miller \cite{MillerLacunary}.
The technical results that are used in the proof are collected in \cref{app:Refine}.
The following statement is a formal reformulation of \cref{thm:E_3dichotomy}

\begin{theorem}\label{thm:E_3version}
Let $G$ be a tsi Polish group, $X$ be a Polish $G$-space and $E^X_G$ be a Borel equivalence relation that is classifiable by countable structures.
Then the following are equivalent:
\begin{enumerate}
	\item $E^X_G$ is essentially countable,
	\item there is a finer $G$-Polish topology on $X$ that is $G$-bounded,
	\item $\I E_0^{\I N}\not\le_B E^X_G$.
\end{enumerate}
\end{theorem}
\begin{proof}[Proof of \cref{thm:E_3dichotomy} and \cref{thm:E_3version}]
(1) $\Rightarrow$ (2) is \cref{cor:EC implies bounded} and (2) $\Rightarrow$ (3) is \cref{cor:bounded implies no E_3}.
It remains to show that (3) $\Rightarrow$ (1).

Recall from \cref{subsec:G_0} that $k_n\in \I N$ is such that $k_0=0$, $k_{n+1}\le \max\{k_m:m\le n\}+1$ for every $n\in\I N$ and for every $k\in\I N$ there are infinitely many $n\in\I N$ such that $k_n=k$.

It will be convenient for us to fix another open base $(V_k)_{k\in \I N}$ at $1_G$.
During the construction we pass two times to a subsequence of $(V_k)_{k\in \I N}$ but we keep the notation $(V_k)_{k\in \I N}$.
In the beginning we set $V_k=\Delta_k$ but one should keep in mind that $(V_k)_{k\in \I N}$ changes during the refinements.
Similar warning applies to the following definition of Miller \cite{MillerLacunary}.
Define $G_{i,j}=R^X_{V_i}\setminus R^X_{V_j}$ for every $i,j\in \I N$.
Suppose that $E^X_G$ does not satisfy (1), i.e., it is not essentially countable.
Then it follows from \cite[Theorem~1.1,~Proposition~2.3]{MillerLacunary} that there is a function $f:\I N\to \I N$ and a continuous homomorphism $\varphi_0:2^\I N\to X$ from $(\I G_{0,k})_{k\in \I N}$ to $(G_{k,f(k)})_{k\in \I N}$.
In this step \cref{thm:G_0Ben} is used.

The proof of the following result can be found in \cref{app:Refine}.
We note that \cref{lm:first ref} uses crucially that $G$ is tsi Polish group while \cref{lm:second ref} holds for every Polish group.

\begin{lemma}[First refinement]\label{lm:first ref}
Let $\varphi_0:2^\I N\to X$ be a continuous homomorphism from $(\I G_{0,k})_{k\in \I N}$ to $(G_{k,f(k)})_{k\in \I N}$.
Then, after possibly passing to a subsequence of $(V_k)_{k\in \I N}$, there is a continuous homomorphism $\varphi_1:2^\I N\to X$ from $(\I G_{s})_{s\in 2^{<\I N}}$ to $(G_{k_{|s|},k_{|s|}+1})_{s\in 2^{<\I N}}$.
\end{lemma}

By \cref{thm:CCS=IC} we have that $E^X_G$ satisfies Property (IC).
That is there is a sequence of Borel sets $(A_{k,l})_{k,l\in \I N}$ such that $\bigcup_{l\in \I N} A_{k,l}=X$ for every $k\in \I N$ and $A_{k,l}$ is $\mathcal{H}_{k,m(k,l)}$-independent for some $m(k,l)\in \I N$.
We stress that the dihypergraphs are defined using the sequence $(\Delta_k)_{k\in \I N}$.
We put $n_0(k)$ to be the minimal $n\in \I N$ such that $k_n=k$.
The proof of the following result can be found in \cref{app:Refine}.

\begin{lemma}[Second refinement]\label{lm:second ref}
Let $\varphi_1:2^\I N\to X$ be a continuous homomorphism from $(\I G_{s})_{s\in 2^{<\I N}}$ to $(G_{k_{|s|},k_{|s|}+1})_{s\in 2^{<\I N}}$.
Then, after possibly passing to a subsequence of $(V_k)_{k\in \I N}$, there is a continuous homomorphism $\varphi:2^\I N\to X$ from $(\I G_{s})_{s\in 2^{<\I N}}$ to $(G_{k_{|s|},k_{|s|}+1})_{s\in 2^{<\I N}}$ and for every $k\in\I N$ there is $m(k)\in \I N$ such that $V_k\subseteq \Delta_{m(k)}$ and the set 
$$\varphi\left(\{s^\frown c\in 2^{\I N}:c\in 2^{\I N}\}\right)$$
is $\mathcal{H}_{k,m(k)}$-independent for every $s\in 2^{<\I N}$ such that $|s|=n_0(k)$.
\end{lemma}

The rest of the proof closely follows the proof of \cite[Theorem~4.1]{MillerLacunary}.
Suppose that we have $\varphi$ as in \cref{lm:second ref}.
In particular it satisfies \cite[Lemma~4.2]{MillerLacunary}.
Set $\ell_n=|\{m<n:k_m=k_n\}|$ for all $n\in\I N$ and define $\psi:2^{\I N\times \I N}\to 2^{\I N}$ as $\psi(c)(n)=c(k_n,\ell_n)$ for all $c\in 2^{\I N\times \I N}$ and $n\in\I N$, see the definition after \cite[Lemma~4.2]{MillerLacunary}.

\begin{claim*}
The continuous map $\varphi\circ\psi:2^{\I N\times \I N}\to X$ is a homomorphism from $\I E_0^{\I N}$ to $E^X_G$.
\end{claim*}
\begin{proof}
Let $c,d\in 2^{\I N\times \I N}$ and suppose that there is $k\in \I N$ such that
$$\left\{(i,\ell)\in \I N\times \I N:c(i,\ell)\not=d(i,\ell)\right\}\subseteq \{k\}\times \I N$$
and the set on the left-hand side is finite.
It is easy to see that $(\varphi\circ\psi(c),\varphi\circ\psi(d))\in E^X_G$.
We show that, in fact, $(\varphi\circ\psi(c),\varphi\circ\psi(d))\in R^X_{\Delta_k}$.

Set $x=\psi(c)$ and $y=\psi(d)$.
Write $(n_1,\dots,n_p)\subseteq \I N$ for the increasing enumeration of the indices where $x$ and $y$ differ.
From the assumption we have $n_0(k)\le n_1$ and $k_{n_i}=k$ for every $1\le i\le p$.
Let $x_1:=x$ and $x_{i+1}$ differ from $x_i$ only in the $n_{i+1}$-th position for every $1\le i<p$.
Clearly, $y=x_p$.
By \cref{lm:first ref} and \cref{lm:second ref}, we have that $(\varphi(x_i),\varphi(x_{i+1}))\in R^{X}_{V_k}\subseteq R^{X}_{\Delta(m(k))}$ and $(\varphi(x_1),\dots,\varphi(x_p))\not\in \fH_{k,m(k)}$.
Consequently, we have
$$(\varphi\circ \psi(c),\varphi\circ\psi(d))=(\varphi(x),\varphi(y))=R^X_{\Delta_k}$$
as promised.

Let $c,d\in 2^{\I N\times \I N}$ be such that $(c,d)\in \I E_0^{\I N}$.
Define $c_0=c$ and $c_{k+1}$ to be equal to $c_k$ except for the vertical section $\{k\}\times \I N$, where it is equal to $d$.
Set $x_k=\varphi\circ \psi(c_k)$, $x=x_0$ and $\varphi\circ \psi(d)=y$
It is easy to see that $c_k\to d$ in $2^{\I N\times {\I N}}$, and by the continuity of $\varphi\circ \psi$ we have $x_k\to y$ in $X$.
Moreover, $c_k,c_{k+1}$ satisfy the assumption above and we conclude that
$$(x_k,x_{k+1})=R^X_{\Delta_k}.$$
Pick $g_k\in \Delta_{k}$ that satisfies $g_k\cdot x_k=x_{k+1}$.
Then we have
$$g_k\cdot g_{k-1}\cdot\dots\cdot g_0\cdot x=x_{k+1}\to y$$
in $X$.
By the definition of $\Delta_k$, we have that
$$g_k\cdot g_{k-1}\cdot\dots\cdot g_0\to h$$
for some $h\in G$.
The continuity of the action guarantees that $h\cdot x=y$, that is, $(x,y)\in E^X_G$, and the proof is finished.
\end{proof}

In order to use \cite[Lemma~3.6]{MillerLacunary} we need to verify the assumptions.
Recall that for $F\subseteq \I N\times \I N$ and $i\in \I N$ we define
$$\I D_{i,F}=\left\{(c,d)\in 2^{\I N\times \I N}\times 2^{\I N\times \I N}:\{(i,n):c(i,n)\not=d(i,n)\}=F\cap (i\times \I N)\right\}$$

\begin{claim*}
For every $i\in \I N$ and finite set $F\subseteq i\times \I N$, the equivalence relation 
$$((\varphi\circ\psi)^{-1}\times (\varphi\circ\psi)^{-1})(E^X_G)$$ is meager in $\I D_{i,F}$.
\end{claim*}
\begin{proof}
The proof of \cite[Lemma~4.5]{MillerLacunary} uses only the fact that the $\varphi$ is a homomorphism from$(\I G_{0,k})_{k\in \I N}$ to $(G_{k,f(k)})_{k\in \I N}$.
Therefore it can be applied in our situation as well.
\end{proof}

The proof is now finished as follows.
By \cite[Lemma~3.6]{MillerLacunary} we find a continuous homomorphism $\phi:2^{\I N\times \I N}\to 2^{\I N\times \I N}$ from $(\I E_0^{\I N},\sim \I E_0^{\I N})$ to $(\I E_0^{\I N},\sim ((\varphi\circ\psi)^{-1}\times (\varphi\circ\psi)^{-1})(E^X_G))$, where $\sim A$ denotes the complement of $A$.
The function $\varphi\circ \psi\circ \phi$ is the desired reduction from $\I E_0^{\I N}$ to $E^X_G$.
\end{proof}

\section{Remarks}

There are two main open questions connected to \cref{thm:CCS=IC} and \cref{thm:E_3dichotomy}.

\begin{question}
Consider the class of Borel equivalence relations induced by actions of tsi Polish groups.
\begin{enumerate}
    \item Let $\mathfrak{C}$ be the collection of meager equivalence relations $E_\Theta$ and $E_\mathcal{Z}$ where $\Theta$ runs over all tall lsc submeasures and $\mathcal{Z}$ over non-trivial sequences of finite metric spaces.
    Is it enough to take $\mathfrak{C}$, instead of $\mathfrak{B}$, as the base of non-classification by countable structures? 
    \item Is the existence of a bounded $G$-Polish topology equivalent to non-reducibility of $\I E_0^{\I N}$?
\end{enumerate}
\end{question}

We conclude our investigation with several remarks.
We mentioned in the introduction that our starting point was Hjorth's dichotomy \cite{HjorthDichotomy}.
From our results, as stated, it does not directly follow that we can recover this dichotomy.
However, there are two easy modifications that gives this result.
First, \cref{thm:CCS=IC} can be stated relative to an analytic set $A\subseteq X$, see~\cite[Theorem~3.3.5]{GrebikPhD}.
Second, if we formulate Hjorth's dichotomy in the form where the Banach space $\ell_1$ is replaced by the summable ideal \cite[Chapter~15]{Kanovei}, then it is not difficult to verify that one is always in (IV) in the proof of \cref{thm:CCS=ICextended}, the lsc submeasure constructed there corresponds to the summable ideal and the map can be refined to a reduction, see \cite[Proof of Theorem~3.1.3]{GrebikPhD}.

As a last thing we mention that it is possible to prove that if $f:Y\to X$ is a Borel reduction from $E$ to $E^X_G$, where $G$ is tsi Polish group, and $E$ admits classification by countable structures, then there is a Borel $G$-nvariant set $B$ that contains $f(Y)$ such that $E^X_G\upharpoonright B\times B$ admits classification by countable structures.
This implies a slight strengthening of \cref{thm:E_3version}.

\begin{theorem}\label{th:main2}
Let $G$ be a tsi Polish group and $X$ be a Polish $G$-space such that $E^X_G$ is Borel.
Then the following are equivalent:
\begin{enumerate}
	\item $E^X_G\upharpoonright B\times B$ is essentially countable for every Borel set $B\subseteq X$ that is $G$-invariant and $E^X_G\upharpoonright B\times B$ is classifiable by countable structures,
	\item $\I E_0^{\I N}\not\le_B E^X_G$.
\end{enumerate}
\end{theorem}

\bibliographystyle{alpha}
\bibliography{main}

\appendix

\section{Technical results}\label{app:technical}

\begin{lemma}\label{lm:weak boundedness}
Let $X$ be a Polish $G$-space and $C\subseteq X$ be a $G$-lg comeager set such that for every $x\in C$ there is an open neighborhood $\Delta$ of $1_G$ such that $\overline{C\cap \Delta\cdot x}\cap C\subseteq [x]_{E^X_G}$.
Then the Polish topology on $X$ is $G$-bounded. 
\end{lemma}
\begin{proof}
The proof consists of three steps.

{\bf (I)}.
Let $x\in C$ and $\Delta$ be any neighborhood of $1_G$.
We show that
$$\overline{C\cap \Delta\cdot x}=\overline{\Delta\cdot x}.$$
Let $y\in \overline{\Delta\cdot x}$.
Then there are $(g_n)_{n\in \I N}\subseteq \Delta$ such that $g_n\cdot x\to y$.
Fix a decreasing sequence $(U_n)_{n\in\I N}$ of open subsets of $X$ such that $\{y\}=\bigcap_{n\in \I N} U_n$.
We may assume (after passing to a subsequence) that $g_n\cdot x\in U_n$ for every $n\in \I N$.
Note that $G(x,U_n)$ is an open subset of $G$ because the action is continuous.
Moreover, we have that the open set $\Delta\cap G(x,U_n)$ is non-empty, since it contains $g_n$.
Now the assumption that $C$ is $G$-lg comeager guarantees that $G(x,C)\cap \Delta\cap G(x,U_n)\not= \emptyset$.
Pick $g'_n\in G(x,C)\cap \Delta\cap G(x,U_n)$.
Then we have $g'_n \in \Delta$ and $g'_n\cdot x\in C\cap U_n$.
This shows that $g'_n\cdot x\to y$ and consequently that $y\in \overline{C\cap \Delta\cdot x}$.

{\bf (II)}.
Let $x\in C$ and $\Delta'\subseteq \Delta$ be open neighborhoods of $1_G$ such that $\overline{C\cap \Delta\cdot x}\cap C\subseteq [x]_{E^X_G}$ and $\Delta'\cdot \Delta'\subseteq \Delta$.
We show that
$$\overline{C\cap \Delta'\cdot x}\subseteq [x]_{E^X_G}.$$
Let $(g_n)_{n\in \I N}\subseteq \Delta'$ be such that $g_n\cdot x\in C$ and $g_n\cdot x\to y$.
Note that $y$ need not be an element of $C$, but since $C$ is $G$-lg comeager set we have $G(y,\Delta')$ is comeager in $\Delta'$.
Pick $g\in G(y,\Delta')$.
Then we have $y_0=g\cdot y\in C$.
It is clearly enough to show that $y_0\in [x]_{E^X_G}$.
We have $g\cdot g_n\cdot x\to g\cdot y=y_0$ because $g$ acts by homeomorphism on $X$, and $g\cdot g_n\in \Delta'\cdot \Delta'\subseteq \Delta$ for every $n\in \I N$.
This shows that $y_0\in \overline{\Delta\cdot x}\cap C$.
By step {\bf (I)} and the assumption we have
$$y_0\in \overline{\Delta\cdot x}\cap C=\overline{C\cap V\cdot x}\cap C\subseteq [x]_{E^X_G}$$
and the claim follows.

{\bf (III)}.
Let $x\in X$.
There is $x_0\in C$ and $g\in G$ such that $g\cdot x_0=x$ because $C$ is $G$-lg comeager.
By {\bf (I)}, {\bf (II)} and the assumption on $C$ there is an open neighborhood $\Delta_0$ of $1_G$ such that $\overline{\Delta_0\cdot x_0}\subseteq [x_0]_{E^X_G}$.
Let $\Delta$ be an open neighborhood of $1_G$ such that $g^{-1}\Delta g\subseteq \Delta_0$.
Let $(g_n)_{n\in \I N}\subseteq \Delta$ be such that $g_n\cdot x\to y\in X$.
Since $g^{-1}$ acts on $X$ by homeomorphism we have that
$$\left(g^{-1}\cdot g_n\cdot g\right)\cdot x_0=g^{-1}\cdot (g_n\cdot x)\to g^{-1}\cdot y.$$
This implies that $g^{-1}\cdot y\in [x_0]_{E^{X}_G}$ and consequently that $y\in [x]_{E^X_G}$.
\end{proof}

\begin{lemma}\label{lm:local continuity}
Let $Y$ be a Polish $H$-space and $X$ be a Polish $G$-space such that $E^X_G$ is Borel.
Suppose that $\varphi:Y\to X$ is a Borel homomorphism from $E^Y_H$ to $E^X_G$, i.e., $(y_0,y_1)\in E^Y_H \ \Rightarrow \ (\varphi(y_0),\varphi(y_1))\in E^X_G$.
Then there is a Borel $H$-lg comeager set $C\subseteq Y$ such that
\begin{itemize}
	\item for every $y\in C$ and every $\Delta\subseteq G$, open neighborhood of $1_G$, there is $\Delta'\subseteq H$, an open neighborhood of $1_H$, such that
$$\varphi(h\cdot y)\in \Delta\cdot \varphi(y)$$
whenever $h\in \Delta'$ and $h\cdot y\in C$.
\end{itemize}
\end{lemma}
\begin{proof}
Fix an open neighborhood $\Delta\subseteq G$ of $1_G$.
Recall that the assumption that that $E^X_G$ is Borel together with \cite[Theorem~7.1.2]{BeckerKechris} gives that $R^X_\Delta$ is a Borel relation.
Put
$$S'=\left(\varphi^{-1}\times \varphi^{-1}\right)\left(R^X_V\right)\subseteq Y\times Y$$
and 
$$S=\left\{(y,h)\in Y\times H:(y,h\cdot y)\in S'\right\}.$$
It follows that $S'$ and $S$ are Borel sets.

Let $(\Delta'_k)_{k\in \I N}$ be an open neighborhood basis at $1_H$.
Define
$$C_{k,\Delta}=\left\{y\in Y:\forall^* h\in \Delta'_k \ (y,h)\in S\right\}.$$
It follows from \cite[Theorem~16.1]{kechrisclassical} that $C_\Delta=\bigcup_{k} C_{k,\Delta}$ is a Borel subset of $Y$.

We show that $C_{\Delta}$ is an $H$-lg comeager set.
To this end pick $y\in Y$ and suppose that $H(y,C_\Delta)$ is not comeager in $H$.
By \cite[Proposition~8.26]{kechrisclassical} there is an open set $U\subseteq H$ such that $H(y,C_\Delta)$ is meager in $U$.
Let $x=\varphi(y)$ and pick a sequence $(g_m)_{m\in \I N}$ such that $\bigcup_{m\in \I N} \tilde{\Delta}\cdot g_m=G$ where $\tilde{\Delta}\cdot (\tilde{\Delta})^{-1}\subseteq \Delta$.
We have $\bigcup_{m\in \I N} \tilde{\Delta}\cdot g_m\cdot x=[x]_{E^X_G}$ and
$$[y]_{E^Y_H}\subseteq \bigcup_{m\in \I N} \varphi^{-1}(\tilde{\Delta}\cdot g_m\cdot x)$$
because $\varphi$ is a homomorphism from $E^Y_H$ to $E^X_G$.
It follows again from \cite[Proposition~8.26]{kechrisclassical} that there is $m\in \I N$ and an open set $U'\subseteq U$ such that $H(y,\varphi^{-1}(\tilde{\Delta}\cdot g_m\cdot x))$ is comeager in $U'$.
Let $h\in H$ and $k\in \I N$ be such that $h\in U'\cap \varphi^{-1}(\tilde{\Delta}\cdot g_m\cdot x)\setminus H(y,C_\Delta)$ and $\Delta'_k\cdot h\subseteq U'$.
We show that $z=h\cdot y\in C_{k,\Delta}$ and that contradicts the choice of $h$.
First note that $A=\{a\in \Delta'_k:a\cdot z\in \varphi^{-1}(\tilde{\Delta}\cdot g_m\cdot x)\}$ is comeager in $\Delta'_k$.
This is because
$$A\cdot h=\Delta'_k\cdot h\cap H(y,\varphi^{-1}(\tilde{\Delta}\cdot g_m\cdot x))$$
and the latter is comaeger in $\Delta'_k\cdot h$ since $\Delta'_k\cdot h$ is an open subset of $U'$.
Pick $a\in A$.
We have $\varphi(z),\varphi(a\cdot z)\in \tilde{\Delta}\cdot g_m \cdot x$ by the definition of $A$ and $h$.
Therefore there are $r,s\in \tilde{\Delta}$ such that
$$s\cdot r^{-1}\cdot \varphi(z)=s\cdot g_m\cdot(r\cdot g_m)^{-1}\varphi(z)=s\cdot g_m\cdot x=\varphi(a\cdot z)$$
This shows that $(\varphi(z),\varphi(a\cdot z))\in R^X_\Delta$ and consequently $(z,a)\in S$.
Altogether, we have that $C_\Delta$ is a Borel $H$-lg comeager set.

Let $(\Delta_i)_{i\in \I N}$ be an open basis at $1_G$ and put
$$C=\bigcap_{i\in \I N} C_{\Delta_i}.$$
It follows from the previous paragraph that $C$ is a Borel $H$-lg comeager set.
Let $y\in C$ and $\Delta\subseteq G$ be an open neighborhood of $1_G$.
Take $i\in\I N$ such that $(\Delta_i)^{-1}\cdot \Delta_i\subseteq \Delta$ and $k\in \I N$ such that $y\in C_{k,\Delta_{i}}$.
Pick $h\in \Delta'_k$ such that $h\cdot y\in C$ and $l\in \I N$ such that $h\cdot y\in C_{l,\Delta_i}$.
We show that $\varphi(h\cdot y)\in \Delta\cdot \varphi(y)$.

Write $A=\{s\in \Delta'_k:\varphi(s\cdot y)\in \Delta_i\cdot \varphi(y)\}$ and $B=\{r\in \Delta'_l:\varphi(r\cdot h\cdot y)\in \Delta_i\cdot \varphi(h\cdot y)\}$.
By the definition we have $y\in C_{k,\Delta_i}$ and $h\cdot y\in C_{l,\Delta_i}$, and consequently $A$ is comeager in $\Delta'_k$ and $B$ is comeager in $\Delta'_l$.
This implies that $A, B\cdot h$ are both comeager in $\Delta'_k\cap \Delta'_l\cdot h$ since $h\in \Delta'_k$ and $\Delta'_k,\Delta'_l$ are open sets.
Let $a\in A\cap B\cdot h$.
Then we have $\varphi(a\cdot y)\in \Delta_i\cdot \varphi(y)$ and $\varphi(a\cdot y)=\varphi((a\cdot h^{-1})\cdot h\cdot y)\in \Delta_i\cdot \varphi(h\cdot y)$.
Consequently, $\Delta_i\cdot \varphi(h\cdot y)\cap \Delta_i\cdot \varphi(y)\not=\emptyset$.
But that implies
$$\varphi(h\cdot y)\in (\Delta_i)^{-1}\cdot \Delta_i\cdot \varphi(y)\subseteq \Delta\cdot \varphi(y)$$
and the proof is finished.
\end{proof}

\begin{lemma}\label{lm:global continuity}
Let $Y$ be a Polish $H$-space and $X$ be a Polish space.
Suppose that $\varphi:Y\to X$ is a Borel map.
Then there is a $H$-lg comeager set $C\subseteq Y$ and a finer $H$-Polish topology $\tau$ on $Y$ such that
\begin{itemize}
	\item $\varphi\upharpoonright C$ is $\tau$-continuous.
\end{itemize}
\end{lemma}
\begin{proof}
Let $U \subseteq X$ be an open set.
We find a finer $H$-Polish topology $\tau_U$ and a Borel $H$-lg comeager set $C_U$ such that $\varphi^{-1}(U)$ is relatively open in $C_U$, i.e., $\varphi^{-1}(U)\cap C_U=O_U\cap C_U$ for some $O_U\in \tau_U$.
Once we have this we finish the proof as follows.
Fix some open basis $(U_r)_{r\in \I N}$ of $X$.
Define $C=\bigcap_{r\in \I N} C_{U_r}$ and $\tau$ to be the topology generated by $\bigcup_{r\in \I N} \tau_{U_r}$.
It is easy to see that $C$ is a Borel $H$-lg comeager set and using \cite[Lemma~4.3.2]{Gao} we have that $\tau$ is a finer $H$-Polish topology on $Y$.
Moreover, for every $r\in \I N$ we have $O_{U_r}\in \tau_{U_r}\subseteq \tau$ and
$$\varphi^{-1}(U_r)\cap C= \varphi^{-1}(U_r)\cap C\cap C_{U_r}= O_{U_r}\cap C \cap C_{U_r}=O_{U_r}\cap C,$$
thus showing that $\varphi\upharpoonright C$ is $\tau$-continuous.

Let $(\Delta_k)_{k\in \I N}$ be an open neighborhood basis at $1_H$.
Define
\begin{itemize}
\item $A_U=\{y\in \varphi^{-1}(U):\exists k\in \I N \ \forall^* h\in \Delta_k \ h\cdot y\in \varphi^{-1}(U)\},$
\item $B_U=\{y\in Y\setminus\varphi^{-1}(U):\exists k\in \I N \ \forall^* h\in \Delta_k \ h\cdot y\in Y\setminus \varphi^{-1}(U)\},$
\item $\widetilde{A}_U=\{y\in Y:\exists k\in \I N \ \forall^* h\in \Delta_k \ h\cdot y\in \varphi^{-1}(U)\},$
\item $\widetilde{B}_U=\{y\in Y:\exists k\in \I N \ \forall^* h\in \Delta_k \ h\cdot y\in Y\setminus\varphi^{-1}(U)\}.$
\end{itemize}
It follows from \cite[Theorem~16.1]{kechrisclassical} that all the sets are Borel.
Moreover, we have $A_U\subseteq \widetilde{A}_U$, $B_U\subseteq \widetilde{B}_U$, $A_U\subseteq \varphi^{-1}(U)$ and $A_U\cap \widetilde{B}_U=\emptyset=B_U\cap\widetilde{A}_U$.
The equalities follow from the fact that if, for example, $y\in A_U\cap \widetilde{B}_U$, then there is $k\in \I N$ such that $H(y,\varphi^{-1}(U))$ and $H(y,Y\setminus \varphi^{-1}(U))$ are both comeager in $\Delta_k$.
This gives $h\in \Delta_k$ such that $h\cdot y\in \varphi^{-1}(U)\cap(Y\setminus \varphi^{-1}(U))$ and that is a contradiction.

Put $C_U=A_U\cup B_U$.
We show that $C_U$ is a Borel $H$-lg comeager set and that there is a finer $H$-Polish topology $\tau_U$ such that $\widetilde{A}_U,\widetilde{B}_U\in \tau_U$.
First we demonstrate how this finishes the proof.
Put $O_U=\widetilde{A}_U$.
Then we have
$$O_U\cap C_U=A_U\subseteq \varphi^{-1}(U)\cap C_U=\varphi^{-1}(U)\cap (A_U\cup B_U)=A_U=O_U\cap C_U$$
by the previous paragraph and the fact that $B_U\cap \varphi^{-1}(U)=\emptyset$.
Hence, $\varphi^{-1}(U)$ is relatively $\tau_U$-open in $C_U$.

First we show that $C_U$ is an $H$-lg comeager set.
Let $y\in Y$ and suppose that $H(y,C_U)$ is not comeager in $H$.
By \cite[Proposition~8.26]{kechrisclassical} there is an open set $W\subseteq H$ such that $H(y,C_U)$ is meager in $W$.
Note that $W\subseteq H(y,\varphi^{-1}(U))\cup H(y,Y\setminus \varphi^{-1}(U))$.
Therefore there is an open set $W'\subseteq W$ such that one of $H(y,\varphi^{-1}(U)), H(y,Y\setminus \varphi^{-1}(U))$ is comeager in $W'$.
Suppose, for example, that $H(y,Y\setminus \varphi^{-1}(U))$ is comeager in $W'$ (the other case is similar). 
Since $W'$ is open we find
$$a\in H(y,Y\setminus \varphi^{-1}(U))\cap W'\setminus H(y,C_U)$$
and $k\in \I N$ such that $\Delta_k\cdot a\subseteq W'$.
The set $P=\{h\in \Delta_k:h\cdot a\cdot y\in Y\setminus \varphi^{-1}(U)\}$ is comeager in $\Delta_k$.
This is because
$$P\cdot a=\Delta_k\cdot a\cap H(y,Y\setminus \varphi^{-1}(U))$$
and the latter set is comeager in $\Delta_k\cdot a$ because $\Delta_{k}\cdot a\subseteq W'$.
Note that $h\cdot a \cdot y\in Y\setminus \varphi^{-1}(U)$ whenever $h\in P$ and $a\cdot y\in Y\setminus \varphi^{-1}(U)$ by the choice of $a$.
We conclude that $a\cdot y \in B_U$ and that is a contradiction with $a\not \in H(y,C_U)$.

Next we show that there is a finer $H$-Polish topology $\tau_U$ such that $\widetilde{A}_U,\widetilde{B}_U\in \tau_U$.

\begin{claim}\label{cl:finer H-top}
Let $R\subseteq Y$ be a Borel set.
Then there is a finer $H$-Polish topology $\sigma$ such that
$$\tilde{R}=\{y\in Y:\exists k\in \I N \ \forall^* h\in \Delta_k \ h\cdot y\in R\}\in \sigma.$$
\end{claim}
\begin{proof}

Recall \cite[Theorem~4.3.3]{Gao}
\begin{itemize}
	\item [(I)] If $Z\subseteq Y$ is a Borel set and $W\subseteq G$ is open, then there is a finer $H$-Polish topology $\sigma'$ such that
	$$\{y\in Y:\exists^* h\in W \ h\cdot y\in Z\}\in \sigma'.$$
\end{itemize}
Define
$$R_k=\{y\in Y:\forall^* h\in \Delta_k \ h\cdot y\in R\}$$
for $k\in \I N$.
Then $R_k$ is a Borel set by \cite[Theorem~16.1]{kechrisclassical} and $\tilde{R}=\bigcup_{k\in \I N} R_k$.
Apply (I) above for $R_k$ and $\Delta_{k+1}$ to find a finer $H$-Polish topology $\sigma_k$ such that
$$S_k=\{y\in Y:\exists^* h\in \Delta_{k+1} \ h\cdot y\in R_k\}\in \sigma_k.$$
Set $\sigma$ to be the topology generated by $\bigcup_{k\in \I N}\sigma_k$.
Then $\sigma$ is a finer $H$-Polish topology by \cite[Lemma~4.3.2]{Gao} and $S=\bigcup_{k\in \I N} S_k\in \sigma$.
To finish the proof we show that $S=\tilde{R}$.

Let $y\in S_k$.
By the definition we have that there is $h\in \Delta_{k+1}$ such that $h\cdot y=z\in R_k$ and consequently that $H(z,R)$ is comeager in $\Delta_k$.
Since $\Delta_{k+1}\cdot (\Delta_{k+1})^{-1}\subseteq \Delta_k$ we have that $\Delta_{k+1}\cdot h^{-1}\subseteq \Delta_k$.
This gives that 
$$\Delta_{k+1}\cdot h^{-1}\cap H(y,R)\cdot h^{-1}=\Delta_{k+1}\cdot h^{-1} \cap H(z,R)$$
is comeager in $\Delta_{k+1}\cdot h^{-1}$ and therefore $H(y,R)$ is comeager in $\Delta_{k+1}$.
Now it is easy to see that $y\in R_{k+1}\subseteq \tilde{R}$.

Suppose now that $y\in R_k$.
Then we have $y\in R_{k+1}$ and $H(y,R)$ is comeager in $\Delta_k$.
Pick $h\in \Delta_{k+1}$.
Then the set
$$\Delta_{k+1}\cdot h\cap H(h\cdot y,R)\cdot h=\Delta_{k+1}\cdot h\cap H(y,R)$$
is comeager in $\Delta_{k+1}\cdot h$ because $\Delta_{k+1}\cdot h\subseteq \Delta_k$ and consequently $H(h\cdot y,R)$ is comeager in $\Delta_{k+1}$.
This shows that $h\cdot y\in R_{k+1}$ for every $h\in \Delta_{k+1}$.
Consequently, $y\in S_{k+1}\subseteq S$ and that finishes the proof.
\end{proof}

Note that we can apply \cref{cl:finer H-top} to both sets $\widetilde{A}_U$ and $\widetilde{B}_U$ to get finer $H$-Polish topologies $\sigma_0$ and $\sigma_1$ such that $\widetilde{A}_U\in \sigma_0$ and $\widetilde{B}_U\in \sigma_1$.
Then by \cite[Lemma~4.3.2]{Gao} we have that the topology $\tau_U$ that is generated by $\sigma_0\cup\sigma_1$ is a finer $H$-Polish topology that contains $\widetilde{A}_U$, $\widetilde{B}_U$ and the proof is finished.
\end{proof}

\begin{lemma}\label{lm:rel open}
Let $X$ be a Polish $G$-space and $(A_k)_{k\in \I N}$ be a sequence of pairwise disjoint Borel subsets of $X$ such that $\bigcup_{k\in \I N}A_k=X$
Then there is a Borel $G$-lg comeager set $C\subseteq X$ such that for every $k\in \I N$ and $x\in C\cap A_k$ there is an open neighborhood $\Delta$ of $1_G$ such that
$$C\cap \Delta\cdot x\subseteq A_k.$$
\end{lemma}
\begin{proof}
Define a Borel map $\varphi:X \to 2^\I N$ where $\varphi(x)(k)=1$ if and only if $x\in A_k$.
Then \cref{lm:global continuity} gives a Borel $G$-lg comeager set $C\subseteq X$ and a finer $G$-Polish topology $\tau$ on $X$ such that $\varphi\upharpoonright C$ is $\tau$-continuous.

Let $k\in \I N$ and $x\in C\cap A_k$.
Since the map $\varphi\upharpoonright C$ is $\tau$-continuous there must be $\tau$-open set $U$ such that $C\cap A_k=U\cap C$.
Note that $\tau$ is a $G$-Polish topology and therefore we find $\Delta$ open neighborhood of $1_G$ such that $\Delta\cdot x\subseteq U$.
Altogether we have $C\cap \Delta\cdot x\subseteq C\cap U=C\cap A_k\subseteq A_k$ and the proof is finished.
\end{proof}

\begin{lemma}\label{pr:Borel independent}
Let $G$ be a tsi Polish group, $X$ be a Polish $G$-space and $A$ be a $\mathcal{H}_{k+2,m}$-independent analytic subset of $X$ for some $k,m\in \I N$.
Then there is a Borel $G$-invariant set $B\subseteq X$ such that $A\subseteq B$ and a sequence $(B_\ell)_{\ell\in \I N}$ of $\mathcal{H}_{k,m+2}$-independent Borel subsets of $X$ such that $\bigcup_{\ell\in \I N} B_\ell=B$.
\end{lemma}
\begin{proof}
We may assume that $k+2\le m$.
Define
$$A'=\left\{x\in X:\exists g\in \Delta_{m+2} \ g\cdot x\in A \right\}.$$
Then it is easy to see that $A'$ is an analytic subset of $X$.
Let $x\in \mathcal{J}(\Delta_{m+2})\cap (A')^{<\I N}$.
Pick any $y\in A^{<\I N}$ such that $|x|=|y|$ and $x_i\in \Delta_{m+2}\cdot y_i$ for every $i<|x|$.
Then we have
$$y_{i+1}\in \Delta^{-1}_{m+2}\cdot x_{i+1}\subseteq \Delta^{-1}_{m+2}\cdot \Delta_{m+2}\cdot  x_i\subseteq \Delta^{-1}_{m+2}\cdot \Delta_{m+2}\cdot \Delta_{m+2}\cdot y_i\subseteq \Delta_m\cdot y_i$$
for every $i<|y|-1$.
The set $A$ is $\mathcal{H}_{k+2,m}$-independent and that gives $y_{|y|-1}\in \Delta_{k+2}\cdot y_0$.
We have
$$x_{|x|-1}\in \Delta_{m+2}\cdot y_{|y|-1}\subseteq \Delta_{m+2}\cdot \Delta_{k+2}\cdot y_0\subseteq \Delta_{m+2}\cdot \Delta_{k+2}\cdot \Delta^{-1}_{m+2}\cdot x_0\subseteq \Delta_{k+1}\cdot x_0$$
and that shows that $A'$ is $\mathcal{H}_{k+1,m+2}$-independent.

By \cite[Theorem~28.5]{kechrisclassical} there is a Borel set $D'\subseteq X$ that is $\mathcal{H}_{k+1,m+2}$-independent and $A'\subseteq D'$.
Define
$$D=\left\{x\in X:\exists r\in \I N \ \forall^*g\in \Delta_r \ g\cdot x\in D'\right\}.$$
It follows from \cite[Theorem~16.1]{kechrisclassical} that $D$ is a Borel set.
The definition of $A'$ together with $A'\subseteq D'$ implies that $A\subseteq D$.
Similar argument as in previous paragraph shows that $D$ is $\mathcal{H}_{k,m+2}$-independent.
Moreover it is easy to see that if $G(x,D')$ is comeager in $\Delta_r$, then $y\in D$ for every $y\in \Delta_{r+1}\cdot x$.
This shows that $G(x,D)$ is open in $G$ for every $x\in X$.

Let $(g_n)_{n\in \I N}$ be a dense subset of $G$ such that $g_0=1_G$.
Define $B_n=g_n\cdot D$ and $B=\bigcup_{n\in \I N} B_n$.
Then $B$ is a $G$-invariant Borel set because $G(x,D)$ is nonempty open set whenever $x\in D$.
Moreover, $A\subseteq D=B_0\subseteq B$.

It remains to show that $B_n$ is $\mathcal{H}_{k,m+2}$-invariant for every $n\in \I N$.
Let $g\in G$, $\Delta$ be a conjugacy invariant open neighborhood of $1_G$ and $x,y\in X$.
Then $y\in \Delta\cdot x$ if and only if $g\cdot y\in \Delta\cdot (g\cdot x)$.
This shows that
$$g_n\cdot (\mathcal{J}(\Delta_{m+2})\cap D^{<\I N})=\mathcal{J}(\Delta_{m+2})\cap B^{<\I N}_n$$
where the action is extended coordinate-wise.
Consequently, $B_n$ is $\mathcal{H}_{k,m+2}$-independent for every $n\in \I N$.
This finishes the proof.
\end{proof}

\section{Refinements}\label{app:Refine}

In this section we prove \cref{lm:refinement}, \cref{lm:first ref} and \cref{lm:second ref}.
Our aim is to develop a technical machinery for finding subtrees of a given finitely uniformly branching tree that satisfy several constraints.
The techniques involve diagonalizing a sequence of trees and iterative application of Baire category argument.
First we define all the relevant notation and then prove two auxiliary lemmata.
These are then used in the proof of our main technical results.

Let $T$ be a finitely uniformly branching tree.
Let $(A,\alpha)\in [\I N]^{\I N}\times [T]$, where $[\I N]^\I N$ denotes the set of all infinite subsets of $\I N$.
We define $T_{(A,\alpha)}\subseteq T$ as
$$s\in T_{(A,\alpha)} \ \Leftrightarrow \ \forall n\not \in A \ s(n)=\alpha(n)$$
and denote as $[T_{(A,\alpha)}]$ the branches of $T_{(A,\alpha)}$.
Note that $[T_{(A,\alpha)}]$ is closed in $[T]$.

Write $(n_l)_{l\in\I N}=A$ for the increasing enumeration of $A$.
Then there is a unique finitely uniformly branching tree $S=S_{(A,\alpha)}$ and a unique map $e_{(A,\alpha)}:S\to T_{(A,\alpha)}$ that satisfy
\begin{itemize}
	\item $\ell^S_l=\ell^T_{n_l}$ for every $l\in \I N$,
	\item $|e_{(A,\alpha)}(s)|=n_{|s|}$
	\item $e_{(A,\alpha)}(s)(n_l)=s(l)$ for every $l<|s|$,
	\item $e_{(A,\alpha)}(s)(j)=\alpha(j)$ for every $j<n_{|s|}$ such that $j\not\in A$.
\end{itemize}
It is easy to verify that $e_{(A,\alpha)}$ extends to a unique continuous homeomorphism
$$\widetilde{e}_{(A,\alpha)}:[S]\to [T_{(A,\alpha)}]$$
that is a reduction from $\I G^S_s$ to $\I G^T_{e_{(A,\alpha)}(s)}$ for every $s\in S$.
This is because for every $l\in \I N$ and $t,t'\in S$ we have $t(l)=t'(l)$ if and only if $e_{(A,\alpha)}(t)(j)=e_{(A,\alpha)}(t')(j)$ for every $n_l\le j<n_{l+1}$.

\begin{lemma}\label{lm:colimit}
Let $(T_r)_{r\in \I N}$ be a sequence of finitely uniformly branching trees, $(A_r,\alpha_r)\in [\I N]^\I N\times [T_r]$ be such that $A_r\cap (r+1)=r+1$ for every $r\in \I N$ and $S_{(A_r,\alpha_r)}=T_{r+1}$ for every $r\in \I N$.
Then there is a finitely uniformly branching tree $S$ and a sequence of continuous maps $(\widetilde{\psi}_{r,\infty}:[S]\to [T_r])_{r\in \I N}$ such that 
\begin{enumerate}
	\item $\ell^S_r=\ell^{T_{r'}}_r$ for every $r\le r'\in \I N$, in particular, $S\cap \I N^{r+1}=T_r\cap \I N^{r+1}$ holds for every $r\in \I N$,
	\item for every $s\in S\cap \I N^{r+1}$ and $x\in [S_s]$ there is $y\in [(T_r)_s]$ such that $\widetilde{\psi}_{r,\infty}(t^\frown x)=t^\frown y$ whenever $t\in S\cap \I N^{r+1}$ for every $r\in \I N$,
	\item $\widetilde{\psi}_{r,\infty}=\widetilde{e}_{(A_r,\alpha_r)}\circ \widetilde{\psi}_{r+1,\infty}$ for every $r\in \I N$,
	\item $\widetilde{\psi}_{r,\infty}$ is a reduction from $\I G^S_s$ to $\I G^{T_r}_s$ for every $s\in S\cap \I N^r$.
\end{enumerate}
\end{lemma}
\begin{proof}
Observe that if $r\le r'\in \I N$, then $\ell^{T_{r'}}_r=\ell^{T_r}_r$.
Define $\ell^S_r:=l^{T_r}_r$ and note that this defines a finitely unifmormly branching tree $S$ that satisfies (1).

For $s\in S\cap \mathbb{N}^r$ and $r\le r'\in \I N$ we define $\psi_{r',\infty}(s)=s$.
For $0\le r'<r$ we set inductively $\psi_{r',\infty}(s)=e_{(A_r,\alpha_r)}\circ \psi_{r'+1,\infty}$.
It is easy to see that $\psi_{r,\infty}=e_{(A_r,\alpha_r)}\circ \psi_{r+1,\infty}$ for every $r\in \I N$ and if $s\sqsubseteq t\in S$, then $\psi_{r,\infty}(s)\sqsubseteq \psi_{r,\infty}(t)$ for every $r\in \I N$.

Define 
$$\widetilde{\psi}_{r,\infty}(x)=\bigcup_{l\in \I N} \psi_{r,\infty}(x\upharpoonright l)$$
for every $x\in [S]$ and $r\in \I N$.
Note that $\widetilde{\psi}_{r,\infty}(x)$ is well defined element of $[T_r]$.
Moreover, we have
\begin{equation*}
    \begin{split}
        \widetilde{\psi}_{r,\infty}(x)= & \bigcup_{l\in \I N} \psi_{r,\infty}(x\upharpoonright l)=\bigcup_{l\in \I N}e_{(A_r,\alpha_r)}\circ \psi_{r+1,\infty}(x\upharpoonright l) \\
        = & \widetilde{e}_{(A_r,\alpha_r)}\left(\bigcup_{l\in \I N} \psi_{r+1,\infty}(x\upharpoonright l)\right)=(\widetilde{e}_{(A_r,\alpha_r)}\circ \widetilde{\psi}_{r+1,\infty})(x)
    \end{split}
\end{equation*}
for every $x\in [S]$ and $r\in \I N$.
This shows (3).

Note that (1) and (2) imply (4) and therefore it remains to show (2).
Let $s\in S\cap \I N^{r+1}$ and $x\in [S_s]$.
Put $y\in [(T_r)_s]$ such that
$$\widetilde{\psi}_{r,\infty}(s^\frown x)=s^\frown y.$$
Let $t\in S\cap \I N^{r+1}$ and $r+1< l\in \I N$.
It is clearly enough to show that $\psi_{r,\infty}(s^\frown x\upharpoonright l)(j)=\psi_{r,\infty}(t^\frown x\upharpoonright l)(j)$ for every $r+1\le j<l$.

We show inductively that $\psi_{r',\infty}(s^\frown x\upharpoonright l)(j)=\psi_{r',\infty}(t^\frown x\upharpoonright l)(j)$ for every $r+1\le j<l$, where $r\le r'\le l$.
By the definition we have 
$$\psi_{l,\infty}(s^\frown x\upharpoonright l)(j)=(s^\frown x\upharpoonright l)(j)=(t^\frown x\upharpoonright l)(j)=\psi_{l,\infty}(t^\frown x\upharpoonright l)(j)$$
for every $r+1\le j<l$.
Suppose that the claim holds for $r'+1$ where $r\le r'<l$.
Fix an enumeration $(m_p)_{p\in \I N}$ of $A_{r'}$.
Then for every $r+1\le j<l$ there is $p\in \I N$ such that $r\le p< l$ and $m_p\le j<m_{p+1}$.
This is because $A_{r'}\cap r+1=r+1$.
If $m_p=j$, then $m_p\not= r\not = p$ and we have
\begin{equation*}
    \begin{split}
        \psi_{r',\infty}(s^\frown x\upharpoonright l)(j)= & \left((e_{(A_{r'},\alpha_{r'})}\circ \psi_{r'+1,\infty})(s^\frown x\upharpoonright l)\right)(m_p)=\psi_{r'+1,\infty}(s^\frown x\upharpoonright l)(p) \\
        = & \psi_{r'+1,\infty}(t^\frown x\upharpoonright l)(p)=\left(e_{(A_{r'},\alpha_{r'})}\circ \psi_{r'+1,\infty}(t^\frown x\upharpoonright l)\right)(m_p) \\
        = & \psi_{r',\infty}(t^\frown x\upharpoonright l)(j)
    \end{split}
\end{equation*}
by the inductive assumption.
If $m_p<j$, then
\begin{equation*}
    \begin{split}
        \psi_{r',\infty}(s^\frown x\upharpoonright l)(j)= & \left((e_{(A_{r'},\alpha_{r'})}\circ \psi_{r'+1,\infty})(s^\frown x\upharpoonright l)\right)(j)=\alpha_{r'}(j) \\
        = & \left((e_{(A_{r'},\alpha_{r'})}\circ \psi_{r'+1,\infty})(t^\frown x\upharpoonright l)\right)(j)=\psi_{r',\infty}(t^\frown x\upharpoonright l)(j)
    \end{split}
\end{equation*}
and the proof is finished.
\end{proof}

\begin{lemma}\label{lm:main technical}
Let $T$ be a finitely uniformly branching tree, $(\mathcal{A}_k)_{k\in \I N}\subseteq [\I N]^\I N$ be such that $\mathcal{A}_k\cap \mathcal{A}_l=\emptyset$ for every $k\not=l$, ${\bf m},{\bf k}\in \I N$, ${\bf p}\in T\cap \I N^{\bf m}$, $(X_r)_{r\in \I N}$ be a sequence of subsets of $[T]$ with the Baire property such that $\bigcup_{r\in \I N} X_r=[T]$ and $(s_n)_{n\in\I N}\subseteq T$ be such that $\{s_n:n\in \mathcal{A}_k\}$ is dense in $T$ for every $k\in \I N$.
Then there is $(A,\alpha)\in [\I N]^\I N\times [T]$ such that, if we put $S=S_{(A,\alpha)}$, we have
\begin{enumerate}
	\item $A\cap {\bf m}={\bf m}$,
	\item for every $s\in S\cap \I N^{\bf m}$ there is $r\in \I N$ such that $s^\frown [S_s]\subseteq (\widetilde{e}_{(A,\alpha)})^{-1}(X_r)$,
	\item $\{v\in S:\exists n\in \mathcal{A}_k \ e_{(A,\alpha)}(v)=s_n\}$ is dense in $S$ for every $k\in \I N$,
	\item there is $n\in \mathcal{A}_{\bf k}$ such that ${\bf p}\sqsubseteq e_{(A,\alpha)}({\bf p})=s_n$.
\end{enumerate}
\end{lemma}
\begin{proof}
Let $(p_l)_{l\in \I N}$ be an enumeration of $T$ such that $\{l\in \mathcal{A}_k:s=p_l\}$ is infinite for every $k\in \I N$ and $s\in T$.
The construction proceeds by induction on $l\in \I N$.
Namely, in every step we construct $t_l\in \I N^{<\I N}$, $n_l\in \I N$, $\alpha_l\in T$ and $S_l\subseteq T$ such that $n_l=|\alpha_l|$,
$$\alpha_l={\bf p}^\frown {t_0} ^\frown (0)^\frown {t_1}^\frown (0)^\frown \dots ^\frown (0)^\frown t_l$$
and
$$S_l=\{s\in T\cap \I N^{n_l+1}: \forall {\bf m}\le j< n_l \ (\forall l'\le l \ j\not=n_{l'} \to s(j)=\alpha_l(j)\}.$$
In the end we put $\alpha=\bigcup_{l\in \I N}\alpha_l$ and $A={\bf m}\cup \{n_l\}_{l\in \I N}$.

{\bf (I)} $l=0$.
Let $(u_i)_{i<N_0}$ be an enumeration of $T\cap \I N^{\bf m}$.
Define inductively $v_i\in \I N^{<\I N}$ such that
\begin{itemize}
    \item ${u_i}^\frown v_i\in T$ for every $i< N_0$,
	\item $v_i\sqsubseteq v_{i+1}$ for every $i<N_0-1$,
	\item for every $i< N_0$ there is $r(i)\in \I N$ such that $X_{r(i)}$ is comeager in ${u_i}^\frown {v_i}^\frown [T_{{u_i}^\frown v_i}]$.
\end{itemize}
This can be achieved by \cite[Proposition~8.26]{kechrisclassical}.
Write $v=v_{N_0-1}$ and use the density of $\{s_n:n\in \mathcal{A}_{\bf k}\}$ to find $n\in \I N$ such that ${\bf p}^\frown v\sqsubseteq s_n$.
Let $t_0\in \I N^{<\I N}$ be such that $\alpha_0={\bf p}^\frown t_0=s_n$ and $n_0=|{\bf p}^\frown t_0|$.

Define
$$X=\bigcup_{i<N_0} {u_i}^\frown {t_0}^\frown [T_{{u_i}^\frown t_0}]\cap X_{r(i)}.$$
Note that $X$ is comeager in ${u_i}^\frown {t_0}^\frown [T_{{u_i}^\frown t_0}]$ for every $i<N_0$.
Fix $\{\mathcal{O}_l\}_{l\in \I N}$ a decreasing collection of open subsets of $[T]$ such that $\mathcal{O}_0=[T]$, $\bigcap_{l\in \I N} \mathcal{O}_l\subseteq X$ and $\mathcal{O}_l$ is dense in ${u_i}^\frown {t_0}^\frown [T_{{u_i}^\frown t_0}]$ for every $i< N_0$.

{\bf (II)} $l\mapsto l+1$.
Suppose that we have $(n_m)_{m\le l}$, $(\alpha_m)_{m\le l}$, $(S_m)_{m\le l}$ and $(t_m)_{m\le l}$ that satisfies
\begin{itemize}
	\item [(a)] $|\alpha_m|=n_m$ and $\alpha_m={\bf p}^\frown {t_0}^\frown (0)^\frown  \dots ^\frown (0)^\frown{t_m}$ for every $m\le l$,
	\item [(b)] $u^\frown [T_u]\subseteq \mathcal{O}_l$ for every $u\in S_l$,
	\item [(c)] if $m<l$, $m\in \mathcal{A}_k$ and $p_m\sqsubseteq u$ for some $u\in S_m$, then there is $n\in \mathcal{A}_k$ such that $p_m\sqsubseteq s_n\in T\cap \I N^{n_{m+1}}$ and $s_n(j)=\alpha_m(j)$ for every $j\not \in {\bf m}\cup \{n_m\}_{m<l}$ such that $j<n_{m+1}$.
\end{itemize}
Note that if $l=0$, then (a)--(c) are satisfied.
Next we show how to find $t_{l+1}\in \I N^{<\I N}$, $\alpha_{l+1}$, $S_{l+1}$ and $n_{l+1}\in \I N$ such that (a)--(c) holds.

Let $(u_i)_{i<N_l}$ be an enumeration of $S_l$.
Construct inductively $(v_i)_{i<N_l}$ such that
\begin{itemize}
	\item $v_i\sqsubseteq v_{i+1}$ for every $i<N_l-1$,
	\item ${u_i}^\frown {v_i}^\frown [T_{{u_i}^\frown {v_i}}]\subseteq \mathcal{O}_{l+1}$ for every $i<N_l$.
\end{itemize}
This can be done because for every $i<N_l$ there is $u\in T\cap \I N^{\bf m}$ such that $u^\frown {t_0}\sqsubseteq u_i$ and $\mathcal{O}_l$ is dense in $u^\frown {t_0}^\frown [T_{u^\frown {t_0}}]$.
Put $v=v_{N_l-1}$.
If $p_l$ satisfies the assumption of (c), that is $p_l\sqsubseteq u\in S_l$ for some $u\in S_l$, and $l\in \mathcal{A}_k$, then pick $i<N_l$ such that $p_l\sqsubseteq u_i$.
Otherwise pick any $i<N_l$.
It follows from the density of $\{s_n:n\in \mathcal{A}_k\}$ that there is $n\in \I N$ such that ${u_i}^\frown v\sqsubseteq s_n$.
Define $t_{l+1}\in \I N^{<\I N}$ such that ${u_i}^\frown t_{l+1}=s_n$, $\alpha_{l+1}={\alpha_l}^\frown (0)^\frown t_{l+1}$, $n_{l+1}=|{u_i}^\frown t_{l+1}|$ and $S_{l+1}=\{u^\frown t_{l+1}^\frown (j):u\in S_l, \ j<\ell^T_{n_{l+1}}\}$.

It is easy to see that (a) and (c) hold.
Property (b) follows from ${u_i}^\frown v_i\sqsubseteq {u_i}^\frown t_{l+1}$ for every $i<N_l$.

{\bf (III)}.
Let $A={\bf m}\cup \{n_l\}_{l\in \I N}$ and $\alpha=\bigcup_{l\in \I N} \alpha_l$.
Set $S=S_{(A,\alpha)}$.
We show that properties (1)--(4) are satisfied.

(1)
Is trivial.

(2)
Let $s\in S\cap \I N^{\bf m}$.
Note that $e_{(A,\alpha)}(s)=s^\frown t_0$ by the definition of $e_{(A,\alpha)}$.
Consequently,
$$\widetilde{e}_{(A,\alpha)}(s^\frown [S_s])\subseteq s^\frown {t_0}^\frown [T_{s^\frown t_0}].$$
By {\bf (I)}, there is $r\in \I N$ such that
$$X\cap s^\frown {t_0}^\frown [T_{s^\frown t_0}]\subseteq X_r$$
and $X$ is comeager in $s^\frown {t_0}^\frown [T_{s^\frown t_0}]$.
Let $c\in [T_s]$ and $l\in \I N$.
Define 
$$u=s^\frown {t_0}^\frown (c(0))^\frown {t_1}^\frown \dots ^\frown (c(l-1))^\frown {t_l}^{\frown}c(l).$$
Then it is easy to see that $u\in S_l$ and, using (b) from the inductive construction, we get
$$\widetilde{e}_{(A,\alpha)}(s^\frown c)\in u^\frown [T_u]\subseteq \mathcal{O}_l.$$
Therefore
$$\widetilde{e}_{(A,\alpha)}(s^\frown c)\in s^\frown {t_0}^\frown [T_{s^\frown {t_0}}]\cap \bigcap_{l\in \I N} \mathcal{O}_l\subseteq X_{r}$$
and that shows (2).

(3)
Let $k\in \I N$ and $s^\frown u\in S$ where $|s|={\bf m}$.
By the properties of the enumeration, there are infinitely many $l\in \mathcal{A}_k$ such that $p_l=e_{(A,\alpha)}(s^\frown u)$.
Pick one such that $|p_l|\le n_l$.
Then during the construction in {\bf (II)} we take $i<N_l$ such that $p_l\sqsubseteq u_i$ and $n\in \I N$ such that $u_i\sqsubseteq s_n$ and $n\in \mathcal{A}_k$ to define $t_{l+1}$ such that ${u_i}^\frown t_{l+1}=s_n$.
Let
$$v=s^\frown (s_n(n_0))^\frown (s_n(n_1))^\frown \dots ^\frown (s_n(n_{l}))\in S.$$
Then it follows from the definition that $e_{(A,\alpha)}(v)=s_n$ and the fact that $e_{(A,\alpha)}(s^\frown u)=p_l\sqsubseteq u_i\sqsubseteq s_n=e_{(A,\alpha)}(v)$ gives $s^\frown u\sqsubseteq v$.

(4)
Note that $e_{(A,\alpha)}({\bf p})={\bf p}^\frown t_0=s_n$ where $n\in \I N$ is such that $n\in \mathcal{A}_{\bf k}$ by the definition in {\bf (I)}.
This finishes the proof.
\end{proof}

\begin{proof}[Proof of \cref{lm:refinement}]
Let $(g_a)_{a\in \I N}$ be a dense subset of $G$.
The construction proceeds by induction on $r\in \I N$.
Let $(p_r)_{r\in \I N}$ be an enumeration of $\I N^{<\I N}$ such that $|\{r\in \I N:p_r=s\}|=\aleph_0$ for every $s\in \I N^{<\I N}$.
We construct a sequence of finitely uniformly branching trees $(T_r)_{r\in \I N}$ together with $(A_r,\alpha_r)\in [\I N]^\I N\times [T_r]$ such that $S_{(A_r,\alpha_r)}=T_{r+1}$ for every $r\in \I N$, $(\mathcal{A}^r)_{r\in \I N}\subseteq [\I N]^\I N$, $(s^r_n)_{n\in \mathcal{A}^r}\subseteq T_r$ for every $r\in \I N$ and $(\varphi_r:[T_r]\to X)_{r\in \I N}$ such that the following holds
\begin{enumerate}
	\item $A_r \cap (r+1)=r+1$ for every $r\in \I N$,
	\item $\varphi_r=\varphi\circ \widetilde{e}_{(A_0,\alpha_0)}\circ \dots \widetilde{e}_{(A_{r-1},\alpha_{r-1})}$ is a homomorphism from $\mathbb{E}^{T_r}_0$ to $E^X_G$ for every $r\in \I N$,
	\item $\{0,\dots,r\}\in \mathcal{A}^r$ for every $r\in \I N$,
	\item $(s^r_n)_{n\in\mathcal{A}^r}$ is a dense subset of $T_r$ such that $|s^r_n|=n$ and $\varphi_r$ is a homomorphism from $\I G^{T_r}_{s^r_n}$ to $\mathcal{H}_{k,n}$ for every $r,n\in \I N$,
	\item if $p_r\in T_r\cap \I N^{< r}$, then $p_r\sqsubseteq s^{r+1}_{r+1}$ (note that $p_r\in T_{r+1}$ by (1)),
	\item for every $s\in T_{r}\cap \I N^r$ there is $g^{s,r}\in G$ such that for every $c\in s^\frown [(T_{r})_s]$ there is $g^{s,r}_c\in G$ such that we have
	\begin{equation*}
	    \begin{split}
	        |d(g^{s,r},1_G)-{\bf d}_{\varphi_{r}}({s^{r}_r}^\frown c,s^\frown c)|< & \ \frac{1}{2^{r+2}}, \\
	       	g^{s,r}_c\cdot \varphi_{r}({s^{r}_r}^\frown c)= & \ \varphi_r(s^\frown c), \\
	       	d(g^{s,r},g^{s,r}_c)< & \ \frac{1}{2^{r+2}}
	    \end{split}
	\end{equation*}
	for every $r\in \I N$, where ${\bf d}_{\varphi_r}$ is defined as in \cref{pr:Borel pseudo}.
\end{enumerate}

{$\bf r=0$}
We put $T_0=T'$, $\mathcal{A}^0=\I N$, $s^0_m=s'_m$ for every $m\in \I N$ and $\varphi_0=\varphi'$.
Conditions (1) and (5) are empty, (2)--(4) are satisfied by assumption and for (6) it is enough to take $g^{\emptyset, 0}=g^{\emptyset,0}_c=1_G$ for every $c\in [T_{\emptyset}]$.

{$\bf r\mapsto r+1$}
In the inductive step we construct $(A_r,\alpha_r)$, $\mathcal{A}^{r+1}$, $(s^{r+1}_{n})_{n\in\mathbb{N}}$ and $\varphi_{r+1}$ such that (1)--(6) holds.
We use a version of \cref{lm:main technical}, where instead of a sequence $(\mathcal{A}_l)_{l\in\I N}$ we take single $\mathcal{A}$, e.g., to apply \cref{lm:main technical} we may consider any partition of $\mathcal{A}$ to disjoint sets and pick ${\bf k}\in \I N$ arbitrarily.

Set $T=T_r$, $\mathcal{A}=\mathcal{A}^r$, ${\bf m}=r+1$, $(s^r_n)_{n\in\mathcal{A}}$, $p_r\sqsubseteq {\bf p}\in T_r\cap \I N^{\bf m}$ if $p_r\in T_r\cap \I N^{< {\bf m}}$ otherwise we put ${\bf p}=(0,\dots,0)\in \I N^{\bf m}$ and $(X_q)_{q\in \I N^{N_r}}$, where $N_r=\{s\in T_r:|s|=r+1\}$ and 
\begin{itemize}
	\item if $s\in N_r$ and $s\not={\bf p}$, then $s^\frown x\in X_q$ for every $q\in \I N^{N_r}$ and $x\in [(T_r)_s]$,
	\item if $x\in [(T_r)_{\bf p}]$, then ${\bf p}^\frown x\in X_q$ if and only if 
	$$\forall s\in N_r \ \left(\exists g^s_x \in G \ d(g^s_x,g_{q(s)})<\frac{1}{2^{r+2}} \ \wedge \ g^s_x\cdot \varphi_r({\bf p}^\frown x)=\varphi_r(s^\frown x)\right) \ \wedge$$
$$\wedge \ |d(g_{q(s)},1_G)-{\bf d}_{\varphi_r}(s^\frown x,{\bf p}^\frown x)|<\frac{1}{2^{r+2}}.$$
\end{itemize}
It is easy to see that the first line in the second item defines an analytic set and it follows from \cref{pr:Borel pseudo} that the second line defines Borel set.
Altogether, $X_q$ is an analytic subset of $[T_r]$, i.e., it has the Baire property by \cite[Theorem~21.6]{kechrisclassical}, for every $q\in \I N^{N_r}$ and $[T_r]=\bigcup_{q\in \mathbb{N}^{N_r}} X_q$.

\cref{lm:main technical} produces $(A_r,\alpha_r)\in [\I N]^\I N\times [T_r]$.
Define $T_{r+1}=S_{(A_r,\alpha_r)}$, $\varphi_{r+1}=\varphi_r\circ \widetilde{e}_{(A_r,\alpha_r)}$,
$$\mathcal{A}^{r+1}=\{|v|\in T_{r+1}:\exists n\in \mathcal{A}^r \ s^r_n=e_{(A_r,\alpha_r)}(v)\}$$
and $(s^{r+1}_{n})_{n\in \mathcal{A}^{r+1}}$ be an enumeration of $e_{(A_r,\alpha_r)}^{-1}((s^r_n)_{n\in \mathcal{A}^r})$ that satisfies $|s^{r+1}_n|=n$ for every $n\in \I N$.

It is easy to see that (1) and (2) hold.
Note that ${\bf p}=s^{r+1}_{r+1}\in T_{r+1}$ because by \cref{lm:main technical}~(4) we have ${\bf p}\sqsubseteq e_{(A_r,\alpha_r)}({\bf p})=s^r_n$ for some $n\in\mathcal{A}^r$.
This shows (3) and (5) follows from $p_r\sqsubseteq {\bf p}$.
First part of item (4) follows from \cref{lm:main technical}~(3).
Second part follows from the inductive hypothesis and definition of $(s^{r+1}_n)_{n\in \mathcal{A}^{r+1}}$.
Namely, for every $n\in \mathcal{A}^{r+1}$ there is $n'\in \mathcal{A}^r$ such that $e_{(A_r,\alpha_r)}(s^{r+1}_n)=s^r_{n'}$.
Note that $n\le n'$.
Then we have that $\varphi_r$ is a homomorphism from $\I G^{T_r}_{s^r_{n'}}$ to $\mathcal{H}_{{ k},n'}$ and $\widetilde{e}_{(A_r,\alpha_r)}$ is a reduction from $\I G^{T_{r+1}}_{s^{r+1}_n}$ to $\I G^{T_r}_{s^{r}_{n'}}$.
This shows that $\varphi_{r+1}$ is a homomorphism from $\I G^{T_{r+1}}_{s^{r+1}_n}$ to $\mathcal{H}_{{k},n'}\subseteq \mathcal{H}_{{k},n}$ because $n\le n'$.

It remains to show (6).
Recall that ${\bf p}=s^{r+1}_{r+1}$.
It follows from \cref{lm:main technical}~(2) that there is $q\in \mathbb{N}^{N_r}$ such that ${\bf p}^\frown [(T_{r+1})_{\bf p}]\subseteq \widetilde{e}^{-1}_{(A_r,\alpha_r)}(X_q)$.
Let $s\in T_{r+1}$ and define $g^{s,r+1}=g_{q(s)}\in G$.
Take any $c\in [(T_{r+1})_{s}]$.
By the definition of $\widetilde{e}_{(A_r,\alpha_r)}$ we find $d\in [(T_{r})_s]=[(T_{r})_{\bf p}]$ such that
$$\widetilde{e}_{(A_r,\alpha_r)}(s^\frown c)=s^\frown d \ \& \ \widetilde{e}_{(A_r,\alpha_r)}({\bf p}^\frown c)={\bf p}^\frown d.$$
Since ${\bf p}^\frown d\in X_q$ we find $g^s_d\in G$ such that, if we set $g^{s,r+1}_c=g^s_d$, we have
\begin{equation*}
    \begin{split}
        d(g^{s,r+1}_c,g^{s,r+1})=  d(g^{s}_d,g_{q(s)})< & \ \frac{1}{2^{r+2}} \\
        |d(g^{s,r+1},1_G)-{\bf d}_{\varphi_{r+1}}(s^\frown c,{\bf p}^\frown c)|= & \  |d(g_{q(s)},1_G)-{\bf d}_{\varphi_r}(s^\frown d,{\bf p}^\frown d)|< \frac{1}{2^{r+2}} \\
        g^{s,r+1}_c\cdot \varphi_{r+1}({\bf p}^\frown c)= & \  g^{s}_d\cdot \varphi_{r}({\bf p}^\frown d)=\varphi_{r}\circ\widetilde{e}_{(A_r,\alpha_r)}(s^\frown c)=  \varphi_{r+1}(s^\frown c)
    \end{split}
\end{equation*}
by the definition of $X_q$.
That shows (6) an the proof is finished.

{\bf Constructing ${\bf \phi}$}.
\cref{lm:colimit} gives a finitely uniformly branching tree $T$ and a sequence of continuous maps $\left(\widetilde{\psi}_{r,\infty}:[T]\to [T_r]\right)_{r\in \I N}$.
Define $\phi=\varphi_r\circ \widetilde{\psi}_{r,\infty}$ for some, or equivalently (by \cref{lm:main technical}~(3)) any, $r\in \I N$.
Note that $\phi$ is a continuous map and $\phi=\varphi\circ \zeta$ where $\zeta=\widetilde{\psi}_{0,\infty}$.

Define $(s_r)_{r\in \I N}=(s^r_r)_{r\in \I N}$.
It follows from (1) and \cref{lm:colimit}~(1) that $s^r_r=s_r\in T$ for every $r\in \I N$ and $|s_r|=r$.
By (4) and \cref{lm:colimit}~(4) we have that $\varphi$ is a homomorphism from $\I G^T_{s_r}$ to $\mathcal{H}_{{k},r}$ for every $r\in \I N$.
Let $s\in T$.
Then there is $r\ge |s|$ such that $p_r=s$.
It follows by (5) that $s=p_r\sqsubseteq s_{r+1}=s^{r+1}_{r+1}$ and, consequently, $(s_r)_{r\in \I N}$ is dense in $T$.

It remains to show that ${\bf d}_\varphi$ is uniform.
Let $s,t\in T\cap \I N^r$ and $x,y\in [T_s]$.
It follows from \cref{lm:colimit}~(2) that there are $c,d\in [(T_r)_s]$ such that
$$\widetilde{\psi}_{r,\infty}(u^\frown x)=u^\frown c \ \& \ \widetilde{\psi}_{r,\infty}(u^\frown y)=u^\frown d$$
for every $u\in T\cap \I N^r$.

Let $g^{s,r}, g^{s,r}_c, g^{s,r}_d \in G$ be as in (6).
We have
\begin{equation*}
    \begin{split}
        |{\bf d}_{\varphi}({s_r}^\frown x,s^\frown x)-{\bf d}_{\varphi}({s_r}^\frown y,s^\frown y)|= & \ |{\bf d}_{\varphi_{r}}({s^{r}_r}^\frown c,s^\frown c)-{\bf d}_{\varphi_{r}}({s^{r}_r}^\frown d,s^\frown d)| \\
        \le & \  |{\bf d}_{\varphi_{r}}({s^{r}_r}^\frown c,s^\frown c)-d(g^{s,r},1_G)|+|d(g^{s,r},1_G)-{\bf d}_{\varphi_{r}}({s^{r}_r}^\frown d,s^\frown d)|\\
        \le & \  \frac{1}{2^{r+1}}.
    \end{split}
\end{equation*}
Consequently, after doing the same argument for $t$, we obtain
$$|{\bf d}_{\varphi}(t^\frown x,s^\frown x)-{\bf d}_{\varphi}(t^\frown y,s^\frown y)|\le \frac{1}{2^r}.$$

Pick any $g,h\in G$ such that $g\cdot \varphi(s^\frown x)=\varphi(s^\frown y)$ and $h\cdot \varphi({s_r}^\frown x)=\varphi({s_r}^\frown y)$ if they exist.
Then we have
\begin{equation}\tag{**}\label{eq:blabla}
    \begin{split}
        (g^{s,r}_d)^{-1}\cdot g\cdot g^{s,r}_c\cdot \varphi({s_r}^\frown x)= & \ (g^{s,r}_d)^{-1}\cdot g\cdot g^{s,r}_c\cdot \varphi_{r}({s_r}^\frown c)=\varphi_r({s_r}^\frown d)=\varphi({s_r}^\frown y) \\
        g^{s,r}_d\cdot h\cdot (g^{s,r}_c)^{-1}\cdot \varphi(s^\frown x)= & \ g^{s,r}_d\cdot h\cdot (g^{s,r}_c)^{-1}\cdot \varphi_{r}(s^\frown c)=\varphi_r(s^\frown d)=\varphi(s^\frown y)
    \end{split}
\end{equation}
by (6).
The invariance of $d$ gives
$$d((g^{s,r}_d)^{-1}\cdot g\cdot g^{s,r}_c,1_G)=d(g,g^{s,r}_d\cdot (g^{s,r}_c)^{-1})\le d(g,1_G)+d(g^{s,r}_d,g^{s,r}_c)\le d(g,1_G)+\frac{1}{2^{r+1}}$$
where the last inequality follows from
$$d(g^{s,r}_d,g^{s,r}_c)\le d(g^{s,r}_d,g^{s,r})+d(g^{s,r},g^{s,r}_c).$$
Similarly
$$d(g^{s,r}_d\cdot h\cdot (g^{s,r}_c)^{-1},1_G)\le d(h,1_G)+\frac{1}{2^{r+1}}.$$
This implies
$$|{\bf d}_{\varphi}(s^\frown x,s^\frown y)-{\bf d}_\varphi({s_r}^\frown x,{s_r}^\frown y)|\le \frac{1}{2^{r+1}}.$$
Similar argument for $t$ implies that
$$|{\bf d}_{\varphi}(s^\frown x,s^\frown y)-{\bf d}_\varphi(t^\frown x,t^\frown y)|\le \frac{1}{2^{r}}.$$
In the case when such $g,h\in G$ do not exist, then, by \eqref{eq:blabla}, we have 
$${\bf d}_{\varphi}(s^\frown x,s^\frown y)={\bf d}_\varphi(t^\frown x,t^\frown y)=+\infty$$
and trivially
$$|{\bf d}_{\varphi}(s^\frown x,s^\frown y)-{\bf d}_\varphi(t^\frown x,t^\frown y)|\le \frac{1}{2^{r}}.$$
This finishes the proof.
\end{proof}

\begin{proof}[Proof of \cref{lm:first ref}]
Recall that $k_n\in \I N$ is such that $k_0=0$, $k_{n+1}\le \max\{k_m:m\le n\}+1$ for every $n\in\mathbb{N}$ and for every $k\in\I N$ there are infinitely many $n\in\I N$ such that $k_n=k$.
Also, we defined $V_k=\Delta_{k}$ for every $k\in \I N$ and the definitions of $(R_{i,j})_{i,j\in \I N}$ are made with respect to $(V_k)_{k\in \I N}$.
Set $W_k=V_{f(k)+2}$ for every $k\in \I N$.
Then it is easy to see that we have $W_{k}\cdot W_{k}\subseteq V_{f(k)+1}\subseteq V_k$ for every $k\in \I N$.
Define $\widetilde{R}_{k,f(k)}$ as
$$(x,y)\in \widetilde{R}_{k,f(k)} \ \Leftrightarrow \ y\in (V_k \cdot V_k)\cdot x \ \wedge \ y\not \not\in V_{f(k)+1}\cdot x$$
for every $x,y\in X$.
Note that we have $R_{k,f(k)}\subseteq \widetilde{R}_{k,f(k)}\subseteq R_{k-1,f(k)+1}$ for every $k>0$.

The proof consists of two steps.
In the first step {\bf (A)} we find a continuous homomorphism $\varphi:2^\I N\to X$ from $(\I G_s)_{s\in 2^{<\I N}}$ to $(\widetilde{R}_{k_{|s|},f(k_{|s|})})_{s\in 2^{<\I N}}$.
In the second step {\bf (B)} we find a subsequence of $(V_{k})_{k\in \I N}$ and $(A,\alpha)\in [\I N]^\I N\times 2^\I N$ such that $\varphi\circ \widetilde{e}_{(A,\alpha)}$ is a homomorphism from $(\I G_s)_{s\in 2^{<\I N}}$ to $(R_{k_{|s|},k_{|s|}+1})_{s\in 2^{<\I N}}$.

{\bf (A)}.
The construction proceeds by induction on $r\in \I N$.
We construct $((A_r,\alpha_r))_{r\in \I N}\subseteq [\I N]^{\I N}\times2^\I N$ together with $(\mathcal{A}^r_k)_{r,k\in \I N}\subseteq [\I N]^{\I N}$, $(\varphi_r:2^\I N\to X)_{r\in \I N}$ and $(s^r_n)_{r,n\in \I N}\subseteq 2^{<\I N}$ such that the following holds
\begin{enumerate}
	\item $A_r \cap (r+1)=r+1$ for every $r\in \I N$,
	\item $(\mathcal{A}^r_k)_{k\in \I N}$ is a partition of $\I N$ for every $r\in \I N$,
	\item $\varphi_r=\phi_0\circ \widetilde{e}_{(A_0,\alpha_0)}\circ \dots \widetilde{e}_{(A_{r-1},\alpha_{r-1})}$ for every $r\in \I N$,
	\item $\{s^r_n:n\in \mathcal{A}^r_k\}$ is dense in $2^{<\I N}$ for every $r,k\in \I N$,
	\item $\varphi_r$ is a homomorphism from $\I G_{s^r_{n}}$ to $R_{k,f(k)}$ whenever $n\in \mathcal{A}^r_k$ for every $r,k\in \I N$,
	\item $\varphi_{r}$ is a homomorphism from $\I G_s$ to $\widetilde{R}_{k_{|s|},f(k_{|s|})}$ for every $s\in 2^r$ and $r\in \I N$.
\end{enumerate}
Having this we use \cref{lm:colimit} and define $\varphi:2^\I N\to X$ as $\varphi=\varphi_r\circ \widetilde{\psi}_{r,\infty}$ for some, or equivalently any, by \cref{lm:colimit}~(2), $r\in \I N$.
Let $s\in 2^{<\I N}$.
By \cref{lm:colimit}~(3) we have that $\widetilde{\psi}_{|s|,\infty}$ is a reduction from $\I G_s$ to $\I G_s$.
Property (6) then implies that $\varphi=\varphi_{|s|}\circ \widetilde{\psi}_{|s|,\infty}$ is a homomorphism from $\I G_s$ to $\widetilde{R}_{k_{|s|},f(k_{|s|})}$.

Let $r\in \I N$ and suppose that we have $(\mathcal{A}^r_k)_{k\in \I N}$, $\varphi_r$ and $(s^r_n)_{n\in \I N}$ that satisfy (2)--(6).
We show how to construct $(A_r,\alpha_r)$, $(\mathcal{A}^{r+1}_k)_{k\in \mathbb{N}}$, $\varphi_{r+1}$ and $(s^{r+1}_n)_{n\in \I N}$ that satisfy (1)--(6).
In the case $r=0$ we put $\mathcal{A}^0_k=\{n\in \mathbb{N}:k_n=k\}$ and $s^0_n=s_n$ for every $n\in \I N$.
Then properties (2)--(5) follow directly from definitions, while (6) is easy to see once we realize that $\varphi_0=\phi_0$ is a homomorphism from $\I G_{\emptyset}$ to $R_{k_{0},f(k_0)}$ and $R_{k_{0},f(k_0)}\subseteq \widetilde{R}_{k_0,f(k_0)}$.

Let $(g_m)_{m\in \I N}$ be a dense subset of $G$.
To build $(A_{r},\alpha_r)$ we use \cref{lm:main technical} with $(\mathcal{A}^r_k)_{k\in \I N}$, $(s^r_n)_{n\in \I N}$, ${\bf k}=k_{r+1}$, ${\bf m}=r+1$, ${\bf p}=(0,\dots, 0)\in 2^{\bf m}$ and $(X_q)_{q\in \I N^{2^{\bf m}}}$, where
$$X_q=\{{\bf p}^\frown c\in 2^{\I N}:\forall u\in 2^{\bf m} \ \varphi_{r}(u^\frown c) \in W_{k_{r+1}}\cdot g_{q(u)}\cdot \varphi_r({\bf p}^\frown c)\}\cup \bigcup_{{\bf p}\not=u\in 2^{\bf m}} u^\frown 2^\I N.$$
It follows from the density of $(g_m)_{m\in \I N}$ that $\bigcup_{q\in \I N^{2^{\bf m}}} X_q=2^\I N$.
Moreover, by the definition, we have that $X_q$ is an analytic subset of $2^\I N$ for every $q\in \mathbb{N}^{2^{\bf m}}$, thus it has the Baire property by~\cite[Theorem~21.6]{kechrisclassical}.
Now, \cref{lm:main technical} gives $(A_{r},\alpha_r)$ that satisfies \cref{lm:main technical}~(1)--(4).
Next we verify properties (1)--(6).

It is easy to see that Properties (1) and (3) hold when we put $\varphi_{r+1}=\varphi_r\circ \widetilde{e}_{(A_{r},\alpha_r)}$.
Let $(s^{r+1}_n)_{n\in \I N}$ be an enumeration of the set $$\{v\in 2^{<\I N}:\exists m\in \I N \ e_{(A_r,\alpha_r)}(v)=s^r_m\}$$
and $\mathcal{A}^{r+1}_k=\{n\in \I N:\exists m\in \mathcal{A}^r_k \ e_{(A_r,\alpha_r)}(s^{r+1}_n)=s^r_m\}$ for every $k\in \I N$.
Then it follows from \cref{lm:main technical}~(3) that Properties (2) and (4) hold.
Let $s^{r+1}_n$ and $m\in \I N$ be such that $e_{(A_r,\alpha_r)}(s^{r+1}_n)=s^r_m$.
Then we have that $\widetilde{e}_{(A_r,\alpha_r)}$ is a reduction from $\I G_{s^{r+1}_n}$ to $\I G_{s^r_m}$.
By (5) of the inductive assumption we have that $\varphi_r$ is a homomorphism from $\I G_{s^r_m}$ to $R_{k,f(k)}$ where $m\in \mathcal{A}^r_k$.
Then we have that $\varphi_{r+1}=\varphi_r\circ \widetilde{e}_{(A_r,\alpha_r)}$ is a homomorphism from $\I G_{s^{r+1}_n}$ to $R_{k,f(k)}$.
By the definition we have $n\in \mathcal{A}^{r+1}_k$ and that shows (5).

It remains to show (6).
By \cref{lm:main technical}~(2) there is $q\in \I N^{2^{\bf m}}$ such that ${\bf p}^\frown 2^\I N\subseteq (\widetilde{e}_{(A_r,\alpha_r)})^{-1}(X_q)$.
By \cref{lm:main technical}~(4) and the definition of $e_{(A_r,\alpha_r)}$ we have that ${\bf p}\sqsubseteq e_{(A_r,\alpha_r)}({\bf p})=s^r_n$ where $n\in \mathcal{A}_{{\bf k}}=\mathcal{A}_{k_{r+1}}$.
We have that $\varphi_{r+1}$ is a homomorphism from $\I G_{\bf p}$ to $R_{k_{r+1},f(k_{r+1})}$ because $\widetilde{e}_{(A_r,\alpha_r)}$ is a reduction from $\I G_{\bf p}$ to $\I G_{s^r_n}$ and $\varphi_r$ is a homomorphism from $\I G_{s^r_n}$ to $R_{k_{r+1},f(k_{r+1})}$ by (5).
Let $c\in \I N$, then we have $\widetilde{e}_{(A_r,\alpha_r)}({\bf p}^\frown c)\in {\bf p}^\frown 2^\I N\cap X_q$ and it follows from the definition of $\widetilde{e}_{(A_r,\alpha_r)}$ that there is $d\in 2^\I N$ such that
$$\widetilde{e}_{(A_r,\alpha_r)}(u^\frown c)=e_{(A_r,\alpha_r)}(u\upharpoonright r)^\frown (u(r))^\frown d=u^\frown d$$
holds for every $u\in 2^{\bf m}$.
This implies that
$$\varphi_{r+1}(u^\frown c)=\varphi_r\circ \widetilde{e}_{(A_r,\alpha_r)}(u^\frown c)=\varphi_r(u^\frown d),$$
and, by the definition of $X_q$, we have
\begin{equation*}
    \begin{split}
        \varphi_{r+1}(u^\frown c)= \varphi_r(u^\frown d)\in & \ W_{k_{r+1}}\cdot g_{q(u)}\cdot \varphi_r({\bf p}^\frown d) \\
        \in & \ W_{k_{r+1}}\cdot g_{q(u)}\cdot (\varphi_r\circ \widetilde{e}_{(A_r,\alpha_r)})({\bf p}^\frown c) \\
        \in & \  W_{k_{r+1}}\cdot g_{q(u)}\cdot \varphi_{r+1}({\bf p}^\frown c)
    \end{split}
\end{equation*}
for every $u\in 2^{\bf m}$.

Recall that $2^{\bf m}=2^{r+1}$.
Let $u \in 2^{r+1}$ and $c\in 2^{\I N}$.
Pick $h_0\in V_{k_{r+1}}$, $g_0\in G$ and $a,b\in W_{k_{r+1}}$ such that 
\begin{itemize}
	\item $\varphi_{r+1}({\bf p}^\frown (1)^\frown c)=h_0\cdot \varphi_{r+1}({\bf p}^\frown (0)^\frown c)$,
	\item $\varphi_{r+1}(u^\frown (1)^\frown c)=g_0\cdot \varphi_{r+1}(u^\frown (0)^\frown c)$,
	\item $\varphi_{r+1}(u^\frown (0)^\frown c)=a\cdot g_{q(u)}\cdot \varphi_{r+1}({\bf p}^\frown (0)^\frown c)$,
	\item $\varphi_{r+1}(u^\frown (1)^\frown c)=b\cdot g_{q(u)}\cdot \varphi_{r+1}({\bf p}^\frown (1)^\frown c)$.
\end{itemize}
An easy calculation shows that
\begin{equation*}
    \begin{split}
        \varphi_{r+1}({\bf p}^\frown (1)^\frown c)= & \ \left(g^{-1}_{q(u)}\cdot b^{-1}\cdot g_0 \cdot a\cdot g_{q(u)} \right) \cdot \varphi_{r+1}({\bf p}^\frown (0)^\frown c)\\
        \varphi_{r+1}(u^\frown (1)^\frown c)= & \ \left(b\cdot g_{q_{u}} \cdot  h_0\cdot g^{-1}_{q(u)} \cdot a^{-1}\right)\cdot \varphi_{r+1}(u^\frown (0)^\frown c).
    \end{split}
\end{equation*}
Recall that $V_{k_{r+1}}$, $V_{f(k_{r+1})}$ and $W_{k_{r+1}}$ are conjugacy invariant and symmetric.
Then we have
$$\left(b\cdot g_{q_{u}} \cdot  h_0\cdot g^{-1}_{q(u)} \cdot a^{-1}\right)\in b\cdot V_{k_{r+1}}\cdot a^{-1}=b\cdot a^{-1} \cdot V_{k_{r+1}}\subseteq V_{k_{r+1}}\cdot V_{k_{r+1}}.$$
Assume that $g_0\in V_{f(k_{r+1})+1}$ then
\begin{equation*}
    \begin{split}
        \left(g^{-1}_{q(u)}\cdot b^{-1}\cdot g_0 \cdot a\cdot g_{q(u)} \right)\in & \ g^{-1}_{q(u)}\cdot W_{k_{r+1}}\cdot V_{f(k_{r+1})+1}\cdot W_{k_{r+1}}\cdot g_{q(u)}\\
        \in & \ V_{f(k_{r+1})+2}\cdot V_{f(k_{r+1})+1}\cdot V_{f(k_{r+1})+2}\\
        \in & \  V_{f(k_{r+1})}.
    \end{split}
\end{equation*}
The assumption that $(\varphi_{r+1}({\bf p}^\frown (0)^\frown c),\varphi_{r+1}({\bf p}^\frown (1)^\frown c))\in R_{k_{r+1},f(k_{r+1})}$ implies that
$$(\varphi_{r+1}(u^\frown (0)^\frown c),\varphi_{r+1}(u^\frown (1)^\frown c))\in \widetilde{R}_{k_{r+1},f(k_{r+1})},$$
hence we have (6).

{\bf (B)}.
In the first step {\bf (A)} we found a continuous homomorphism $\varphi:2^{\I N}\to X$ from $(\I G_s)_{s\in 2^{<\I N}}$ to $(\widetilde{R}_{k_{|s|},f(k_{|s|})})_{s\in 2^{<\I N}}$.
Put $k^0=0$ and define inductively $k^{i+1}=f(k^i+1)+1$ for every $i\in \I N$.
Let
$$A=\{n\in \I N:\exists i\in \I N \ k_n=k^i+1\}$$
and $\alpha=(0,0,,\dots)\in 2^{\I N}$.
Put $\phi_1=\varphi\circ\widetilde{e}_{(A,\alpha)}:2^\I N\to X$ and define $i_{|s|}\in \I N$ such that $k_{|e_{(A,\alpha)}(s)|}=k^{i_{|s|}}+1$.
This is well defined since $|e_{(A,\alpha)}(s)|\in A$ for every $s\in 2^{<\I N}$.

We show that $\phi_1$ is a homomorphism from $(\I G_s)_{s\in 2^{<\I N}}$ to $(R_{k^{i_{|s|}},k^{i_{|s|}+1}})_{s\in 2^{<\I N}}$.
We have that $\widetilde{e}_{(A,\alpha)}$ is a reduction from $\I G_s$ to $\I G_{e_{(A,\alpha)}(s)}$ for every $s\in 2^{<\I N}$.
Fix $s\in 2^{<\I N}$ and let $n=|e_{(A,\alpha)}(s)|$.
Then $\phi_1$ is a homomorphism from $\I G_s$ to $\widetilde{R}_{k_{n},f(k_{n})}$.
We have 
$$\widetilde{R}_{k_n,f(k_n)} \subseteq R_{k^{i_{|s|}},k^{i_{|s|}+1}}$$
because $V_{k_n}\cdot V_{k_n}\subseteq V_{k_n-1}=V_{k^i}$ and $V_{f(k_n)+1}=V_{f(k^i+1)+1}=V_{k^{i+1}}$.
Therefore, after passing to the subsequence $(V_{k^i})_{i\in \I N}$ we have that $\phi_1$ is a homomorphism from $(\I G_s)_{s\in 2^{<\I N}}$ to $(R_{k^{i_{|s|}},k^{i_{|s|}+1}})_{s\in 2^{<\I N}}$ and that finishes the proof.
\end{proof}

\begin{proof}[Proof of \cref{lm:second ref}]
Recall that $n_0(k)$ is the minimal number such that $k_n=k$ and $(V_i)_{i\in \I N}$ is a decreasing sequence of open neighborhoods of $1_G$.
The relations $(R_{i,j})_{i,j\in \I N}$ are defined with respect to $(V_k)_{k\in \I N}$.
Recall also that $\Delta_k=\{g\in G:d(g,1_G)<\frac{1}{2^k}\}$, where $d$ is some fixed compatible metric on $G$, and the relations $(\fH_{k,m})_{k,m\in\I N}$ are defined with respect to $(\Delta_k)_{k\in \I N}$.
We assume that $V_k\subseteq \Delta_k$ for every $k\in \I N$.
By the assumption, there is a sequence of Borel sets $\{A_{k,l}\}_{k,l\in \I N}$ such that $\{A_{k,l}\}_{l\in\I N}$ is a partition of $X$ for every fixed $k\in \I N$ and $A_{k,l}$ is $\mathcal{H}_{k,m(k,l)}$-independent for every $k,l\in \I N$ and some $m(k,l)\in \I N$.

The proof consists of two steps.
In the first step, {\bf (A)}, we find a continuous homomorphism $\varphi:2^{\I N} \to X$ from $(\mathbb{G}_s)_{s\in 2^{<\I N}}$ to $(R_{k_{|s|},k_{|s|}+1})_{s\in 2^{<\I N}}$ and a sequence $(m(r))_{r\in \I N}\subseteq \I N$ such that
$$\varphi(s^\frown 2^\I N)$$
is $\mathcal{H}_{|s|,m(|s|)}$-independent for every $\emptyset\not= s\in 2^{<\I N}$.
In the second step, {\bf (B)}, we pass to a subsequence of $(V_k)_{k\in \I N}$ and massage $\varphi$ so that it meets the desired requirements.

{\bf (A)}.
The construction proceeds by induction on $r\in \I N$.
We construct $((A_{r},\alpha_r))_{r\in \I N}\subseteq [\I N]^\I N\times 2^{\I N}$, $(\mathcal{A}^r_k)_{r,k\in \I N}\subseteq [\I N]^\I N$, $(\varphi_r:2^\I N\to X)_{r\in \I N}$, $(m(r))_{r\in \I N}\subseteq \I N$ and $(s^r_n)_{r\in \I N}\subseteq 2^{<\I N}$ such that the following holds
\begin{enumerate}
	\item $A_r\cap (r+1)=r+1$ for every $r\in\I N$,
	\item $(\mathcal{A}^r_k)_{k\in\I N}$ is a partition of $\I N$ for every $k\in\I N$,
	\item $\varphi_r=\phi_1\circ \widetilde{e}_{(A_0,\alpha_0)}\circ \dots \widetilde{e}_{(A_{r-1},\alpha_{r-1})}$ for every $r\in \I N$,
	\item $\{s^r_n:n\in \mathcal{A}^r_k\}$ is dense in $2^{<\I N}$ for every $r,k\in \I N$,
	\item $\varphi_r$ is a homomorphism from $\I G_{s}$ to $R_{k,k+1}$ whenever $|s|\in \mathcal{A}^r_k$ for every $r,k\in \mathbb{N}$,
	\item $\varphi_r$ is a homomorphism from $(\I G_s)_{s\in 2^{r}}$ to $(R_{k_{|s|},k_{|s|}+1})_{s\in 2^{r}}$,
	\item $\varphi_{r+1}(s^\frown 2^\I N)$ is $\mathcal{H}_{r+1,m(r+1)}$-independent for every $s\in 2^{r+1}$ and for every $r\in \I N$.
\end{enumerate}
Having this, we use \cref{lm:colimit} and define $\varphi:2^\I N\to X$ as $\varphi=\varphi_r\circ \widetilde{\psi}_{r,\infty}$ for some, or equivalently, by \cref{lm:colimit}~(2) any, $r\in \I N$.
By \cref{lm:colimit}~(3), we have that $\widetilde{\psi}_{r,\infty}$ is a reduction from $\I G_s$ to $\I G_s$ for every $s\in 2^r$ and every $r\in \I N$.
This implies that $\varphi$ is a homomorphism from $(\I G_s)_{s\in 2^{<\I N}}$ to $(R_{k_{|s|},k_{|s|}+1})_{s\in 2^{<\I N}}$.
Let $s\in 2^{r+1}$, then it follows from \cref{lm:colimit}~(1)and~(2) that
$$\widetilde{\psi}_{r+1,\infty}(s^\frown 2^\I N)\subseteq s^\frown 2^\I N.$$
This implies that $\varphi(s^\frown 2^{\I N})$ is $\mathcal{H}_{|s|,m(|s|)}$-independent for every $\emptyset\not =s\in 2^{<\I N}$.

Let $r\in \I N$ and suppose that we have $(\mathcal{A}^r_{k})_{k\in \I N}$, $\varphi_r$ and $(s^r_n)_{n\in \I N}$ that satisfy (2)--(5).
We show how to construct $(A_r,\alpha_r)$, $(\mathcal{A}^{r+1}_{k})_{k\in \I N}$, $\varphi_{r+1}$, $m(r+1)\in \I N$ and $(s^{r+1}_n)_{n\in \I N}$ that satisfy (1)--(7).
In the case when $r=0$ we put $\mathcal{A}^0_k=\{n\in \I N:k_n=k\}$ for every $k\in \I N$, $m(0)=0$, $\varphi_0=\phi_1$ and choose any $(s^0_n)_{n\in\I N}$ that satisfies (4).
Then it is easy to see that properties (2)--(5) are satisfied.

To build $(A_r,\alpha_r)$ we use \cref{lm:main technical} with $(\mathcal{A}^r_{k})_{k\in \I N}$, $(s^r_n)_{n\in\I N}$, ${\bf k}=k_{r+1}$, ${\bf m}=r+1$, ${\bf p}=(0,\dots, 0)\in 2^{\bf m}$ and $(X_{l})_{l\in \I N}$, where 
$$X_l=\varphi^{-1}_r(A_{r+1,l})$$
for every $l\in \I N$.
Now, \cref{lm:main technical} gives $(A_{r},\alpha_r)$ that satisfies \cref{lm:main technical}~(1)--(4).

It is easy to see that Properties (1) and (3) hold when we put $\varphi_{r+1}=\varphi_r\circ \widetilde{e}_{(A_{r},\alpha_r)}$.
Let $(s^{r+1}_n)_{n\in \I N}$ be an enumeration of the set $$\{v\in 2^{<\I N}:\exists m\in \I N \ e_{(A_r,\alpha_r)}(v)=s^r_m\}$$
and $\mathcal{A}^{r+1}_k=\{n\in \I N:\exists s\in 2^n \ |e_{(A_r,\alpha_r)}(s)|\in \mathcal{A}^r_k\}$ for every $k\in \I N$.
It follows from \cref{lm:main technical}~(3) that Properties (2) and (4) hold.
Let $s\in 2^{n}$ be such that $n\in \mathcal{A}^{r+1}_k$ for some $k\in \I N$.
Then by the definition we have $|e_{(A_r,\alpha_r)}(s)|\in \mathcal{A}^r_k$ and $\varphi_r$ is a homomorphism from $\I G_{e_{(A_r,\alpha_r)}(s)}$ to $R_{k,k+1}$ by (5) of the inductive assumption.
This gives that $\varphi_{r+1}$ is a homomorphism from $\I G_s$ to $R_{k,k+1}$ because $\widetilde{e}_{(A_r,\alpha_r)}$ is a reduction from $\I G_s$ to $\I G_{e_{(A_r,\alpha_r)}(s)}$.
This shows (5).
By \cref{lm:main technical}~(4) we have that $|e_{(A_r,\alpha_r)}({\bf p})|\in \mathcal{A}^r_{\bf k}=\mathcal{A}^r_{k_{r+1}}$.
This implies that $r+1\in \mathcal{A}^{r+1}_{k_{r+1}}$ and, by (5), $\varphi_{r+1}$ is a homomorphism from $\I G_{s}$ to $R_{k_{r+1},k_{r+1}+1}$ for every $s\in 2^{r+1}$.
This proves (6).

By \cref{lm:main technical}~(2) we have that for every $s\in 2^{r+1}$ there is $l(s)\in \I N$ such that $s^\frown 2^\I N\subseteq (\widetilde{e}^-1_{(A_r,\alpha_r)})(X_{l(s)})$.
Consequently,
$$\varphi_{r+1}(s^\frown 2^\I N)\subseteq A_{r+1,l(s)}$$
and there is $m(s)=m(r+1,l(s))\in \I N$ such that $\varphi_{r+1}(s^\frown 2^\I N)$ is $\mathcal{H}_{r+1,m(s)}$-independent.
Define
$$m(r+1)=\max_{s\in 2^{r+1}} m(s).$$
Then it is easy to see that $\varphi_{r+1}(s^\frown 2^\I N)$ is $\mathcal{H}_{r+1,m(r+1)}$-independent for every $s\in 2^{r+1}$ and the proof of {\bf (A)} is finished.

{\bf (B)}.
Let $\alpha=(0,0,\dots)\in 2^{\I N}$.
Inductively define an increasing sequence $(k^i)_{i\in \I N}\subseteq \I N$ such that $V_{k^i}\subseteq \Delta_{m(i+1)}$ and $k^0\ge 1$.
Put $A=\{n\in \I N:\exists i\in \I N \ k_n=k^i\}$, define $\phi=\varphi\circ \widetilde{e}_{(A,\alpha)}:2^\I N\to X$ and $i_{|s|}\in \I N$ such that $k_{|e_{(A,\alpha)}(s)|}=k^{i_{|s|}}$ for every $s\in 2^{<\I N}$.

We have that $\varphi$ and, consequently, $\phi$ are homomorphisms from $\I G_{e_{(A,\alpha)}(s)}$ and $\I G_s$, respecetively, to $R_{k_{|e_{(A,\alpha)}(s)|},k_{|e_{(A,\alpha)}(s)|}+1}$ for every $s\in 2^{<\I N}$.
Since we have
$$V_{k^{i_{|s|}+1}}\subseteq V_{k^{i_{|s|}}+1}= V_{k_{|e_{(A,\alpha)}(s)|}+1},$$
we conclude that $\phi$ is a homomorphism from $(\I G_s)_{s\in 2^{<\I N}}$ to $(R_{k^{i_{|s|}},k^{i_{|s|}+1}})_{s\in 2^{<\I N}}$.

Let $n_0(i)$ be the minimal $n\in \I N$ such that $i_{|s|}=i$ for some, or equivalently any, $s\in 2^n$.
Let $s\in 2^{n_0(i)}$.
Then there is unique $v\in 2^{i+1}$ such that $v\sqsubseteq e_{(A,\alpha)}(s)$. 
This is because $k_{|e_{(A,\alpha)}(s)|}=k^i$ and therefore $|e_{(A,\alpha)}(s)|\ge k^i>i$ because $(k^i)_{i\in \I N}$ is increasing and $k^0\ge 1$.
Then we have
$$\phi(s^\frown 2^\I N)\subseteq \varphi(v^\frown 2^\I N)$$
and the latter set is $\mathcal{H}_{i+1,m(i+1)}$-independent.
This implies trivially that $\phi(s^\frown 2^\I N)$ is $\mathcal{H}_{i,m(i+1)}$-independent for every $s\in 2^{n_0(i)}$.
Passing to a subsequence $(V_{k^i})_{i\in \I N}$ and setting $m(i):=m(i+1)$ then work as required.
The proof is finished.
\end{proof}

\end{document}